\documentclass[11pt]{article}
\usepackage{amsmath,amssymb,amsthm,stix}
\usepackage{stmaryrd}
\usepackage{xfrac}
\usepackage{comment}
\allowdisplaybreaks
\usepackage{xcolor}

\usepackage[pagebackref]{hyperref}
\usepackage[utf8]{inputenc}

\begin{document}

\def\fl#1{\left\lfloor#1\right\rfloor}
\def\cl#1{\left\lceil#1\right\rceil}
\def\ang#1{\left\langle#1\right\rangle}
\def\stf#1#2{\left[#1\atop#2\right]}
\def\sts#1#2{\left\{#1\atop#2\right\}}
\def\eul#1#2{\left\langle#1\atop#2\right\rangle}

\def\Z{\mathcal Z}
\def\B{\mathcal B}
\def\set{{\sc Set}}
\def\sset{{\sc Sset}}
\def\setk{{\sc Set}$_k$}
\def\sttf2#1#2{\left[\!\!\left[#1\atop#2\right]\!\!\right]}
\def\stss2#1#2{\left\{\!\!\left\{#1\atop#2\right\}\!\!\right\}}

\def\StirB#1#2{\ensuremath{ S_{ #1 , #2}^B}}

\def\StirBq#1#2{\ensuremath{ S_{ #1 , #2}^B(q)}}

\def\StirBr#1#2{\ensuremath{ S_{ #1 , #2,r}^B}}

\def\StirBqr#1#2{\ensuremath{ S_{ #1 , #2,r}^B(q)}}

\newtheorem{theorem}{Theorem}[section]
\newtheorem{definition}[theorem]{Definition}
\newtheorem{thm}[theorem]{Theorem}
\newtheorem{prop}[theorem]{Proposition}
\newtheorem{example}[theorem]{Example}
\newtheorem{cor}[theorem]{Corollary}
\newtheorem{fact}[theorem]{Fact}
\newtheorem{lemma}[theorem]{Lemma}
\newtheorem{remark}[theorem]{Remark}
\newtheorem{conj}[theorem]{Conjecture}

\newenvironment{Rem}{\begin{trivlist} \item[\hskip \labelsep{\it
Remark.}]\setlength{\parindent}{0pt}}{\end{trivlist}}

\title{Analytical aspects of  $q,r$-analogue of poly-Stirling numbers of both kinds}

\author{
Takao Komatsu
\\
\small Department of Mathematical Sciences, School of Science\\[-0.8ex]
\small Zhejiang Sci-Tech University\\[-0.8ex]
\small Hangzhou 310018, China\\[-0.8ex]
\small \texttt{komatsu@zstu.edu.cn}
\\\vspace{-7pt}\\
Eli Bagno\\
\small Jerusalem College of Technology\\[-0.8ex]
\small 21 HaVaad HaLeumi St.\\[-0.8ex]
\small Jerusalem, Israel\\[-0.8ex]
\small and \\[-0.8ex]
\small Michlalah College Jerusalem\\[-0.8ex]
\small 36 Barukh Duvdevani St.,\\[-0.8ex] \small Jerusalem, Israel\\[-0.8ex]
\small
\texttt{bagnoe@g.jct.ac.il}\\\vspace{-7pt}\\
David Garber\\
\small Holon Institute of Technology\\[-0.8ex]
\small 52 Golomb St., P.O.Box 305\\[-0.8ex]
\small 5810201 Holon, Israel\\[-0.8ex]
\small \texttt{garber@hit.ac.il}
}

\date{
}

\maketitle

\vspace{-16pt}

\begin{abstract}

The Stirling numbers of type $B$ of the second kind count signed set partitions. In this paper we provide new combinatorial and analytical identities regarding these numbers as well as Broder's $r$-version of these numbers. Among these identities one can find recursions, explicit formulas based on the inclusion-exclusion principle, and also exponential generating functions.

These Stirling numbers can be considered as members of a wider family of triangles of numbers that are characterized using results of Comtet and Lancaster.

We generalize these theorems, which present equivalent conditions for a triangle of numbers to be a triangle of generalized Stirling numbers, to the case of the $q,r$-poly Stirling numbers, which are $q$-analogues of the restricted Stirling numbers defined by Broder and having a polynomial value appearing in their defining recursion.
There are two ways to do this and these ways are related by a nice identity.
\end{abstract}


\medskip

\noindent
{\bf MR Subject Classifications:} Primary: 05A15; Secondary: 05A18, 05A19, 05A30, 11B73

\section{Introduction}

The {\em Stirling number of the second kind}, denoted $S_{n,k}$,
counts the number of partitions of the set $[n]:=\{1,\dots,n\}$ into $k$ non-empty subsets (see Stanley \cite[page 81]{EC1}).
Stirling numbers of the second kind arise in a variety of problems  in enumerative combinatorics;
they have many combinatorial interpretations and have been generalized in various contexts and in different ways.

\medskip

The Stirling numbers have also a vast algebraic background and the following recursion can be considered as the link from the combinatorial point of view to its algebraic one:

\begin{equation}\label{kind 2 rec}
S_{n,k}= S_{n-1,k-1} + k S_{n-1,k}.
\end{equation}

\subsection{The Comtet and Lancaster approaches to Stirling numbers}

Comtet \cite{comtet} showed that the Stirling numbers of the second kind can be presented
in several equivalent algebraic ways. We provide here the content of Comtet's theorem, following the presentation
of Wagner \cite[Section 7.2, Theorem 7.2.1]{Wagner-book}:
\begin{thm}\label{Comtet theorem}
Let $(b_n)_{n \geq 0}$ be a sequence of complex numbers.
The following are equivalent characterizations for an array of numbers $(A_{n,k})_{n ,k \geq 0}$:
\begin{enumerate}
\item[(1)] {\bf Defining equation:} For each $n \geq 0$: $$x^n= \sum\limits_{k \geq 0} A_{n,k}\cdot (x-b_0)(x-b_1) \cdots (x-b_{k-1}).$$
\item[(2)] {\bf Recursion:} For each $n \geq k > 0$:
$$A_{n,k}= A_{n-1,k-1} + b_k A_{n-1,k}$$
with the boundary conditions:
$A_{n,0}=b_0^n$ and $A_{0,k} =\delta_{0k}$ for each $n\geq k \geq 0$.
\item[(3)] {\bf Complete recursion:} For $n \geq k > 0$:
$$A_{n,k} = \sum\limits_{j=k}^{n} A_{j-1,k-1}b_k^{n-j},$$
subject to the same boundary conditions as in Condition (2).
\item[(4)] {\bf Ordinary generating function:} For each $k \geq 0$:
$$\sum\limits_{n \geq 0} A_{n,k}x^n= \frac{x^k}{(1-b_0 x) \cdots (1-b_k x)}.$$
\item[(5)] {\bf Explicit formula:} For each $n\geq k \geq 0$:
$$A_{n,k} = \sum\limits_{d_0+d_1+\cdots +d_k=n-k\atop d_i\geq 0} b_0^{d_0} \cdots b_k^{d_k}.$$

\end{enumerate}
\end{thm}

Note that if $b_k=1$ for all $k$, the numbers $A_{n,k}$ are the binomial coefficients $n\choose k$, and if $b_k=k$ for all $k$, then the numbers $A_{n,k}$ are the ordinary Stirling numbers of the second kind.

\medskip

The {\it (unsigned) Stirling number of the first kind}, denoted $c_{n,k}$, counts the number of permutations of the set $[n]$ having $k$ cycles (see Stanley \cite[page 32]{EC1}).
The recursion satisfied by  these numbers is:
\begin{equation}\label{first kind def}
c_{n,k}=c_{n-1,k-1}+(n-1)c_{n-1,k}.
\end{equation}

The {\it (signed) Stirling number of the first kind}, denoted $s_{n,k}$, is defined by the recursion:
\begin{equation}\label{first kind def signed}
s_{n,k}=s_{n-1,k-1}-(n-1)s_{n-1,k},
\end{equation}
and they satisfy some orthogonality relations with the Stirling numbers of the second kind (see e.g. \cite[p. 264]{GKP}).

\medskip

Back to the unsigned Stirling numbers of the first kind, an analogue of Comtet's theorem for these numbers was given by Lancaster \cite{Lancaster}, see also \cite{LMMS} and Wagner's book \cite[Section 7.2]{Wagner-book} (we fix $b_i=0$ in the original formulation of Lancaster):

\begin{theorem}[Lancaster]\label{Lancaster}
Let $(a_n)_{n \geq 0}$ be a sequence of complex numbers.
The following are equivalent characterizations for $(c_{n,k})_{n ,k\geq  0}$:
\begin{enumerate}
\item[(1)] {\bf Defining equation/generating function:} $$(x+a_0)(x+a_1) \cdots (x+a_{n-1}){=} \sum\limits_{k=0}^n c_{n,k}\cdot x^k.$$
\item[(2)] {\bf Recursion:} For each $n\geq k \geq 0$:
$$c_{n,k}= c_{n-1,k-1} + a_{n-1} c_{n-1,k}$$
with the boundary conditions:
$c_{n,0}=a_0 a_1 \cdots a_{n-1}$ and $c_{0,k}=\delta_{0k}$.
\item[(3)] {\bf Complete recursion:} For $n\geq k \geq 0$:
$$c_{n,k} = \sum\limits_{j=k}^{n} c_{j-1,k-1}\prod\limits_{i=j}^{n-1}a_i,$$
subject to the same boundary conditions as in Condition (2).

\end{enumerate}
\end{theorem}

\medskip
5\subsection{Broder's restricted $r$-Stirling numbers}\label{section intro Broder}

Broder \cite{Broder} (see also \cite{d'Ocagne}) defined an $r$-version to both kinds of Stirling numbers, which counts set partitions such that the first $r$ elements are placed in $r$ distinguished parts in the case of the second kind platform, and permutations of $[n]$ which are decomposed in $k$ cycles such that the elements $1,\dots,r$ are in distinguished cycles in the case of the first kind platform.
An excellent textbook that deals with this type of generalization of Stirling numbers is Mez\H{o}'s book \cite{Mezo}.

\medskip

The ordinary Stirling numbers of the second kind can also be interpreted as the number of elements of a constant rank in the intersection lattice of hyperplane arrangements of Coxeter type $A$.
Dolgachev-Lunts \cite[p.~755]{DoLu} and Reiner \cite[Section 2]{R} generalized this idea to hyperplane arrangements of Coxeter type $B$ (in \cite{DoLu}, the partitions of type $B$ are counted by $\tilde S (n,k)_0$ in their notation).
However, the concept of set partitions of type $B$ has already appeared implicitly in Dowling \cite{Dow} and also in Zaslavsky \cite{Za} in the form of signed graphs. The Stirling numbers of the second kind of type $B$
enumerate the set partitions of type $B$;
the exact definitions will be recalled in Section \ref{second kind}.

\medskip

The Stirling numbers were further generalized in various ways. Bala \cite{Bala} defined generalized Stirling numbers which he called $S_{(a,b,c)}$ via their exponential Riordan array. The {\it $r$-Whitney numbers} can be considered as special cases of these numbers for $a=m$ and $c=r$ (see e.g. \cite[Eqn. (1.2)]{MRS}).
In this work, we cover the case $(a,b,c)=(2,0,1)$ in a combinatorial way, see Corollary \ref{coro-for-bala}  below in which we derive a special case of Equation (4) of Bala \cite{Bala} using our methods. The numbers of the form $S_{(2,0,r)}$ were granted a combinatorial interpretation by Corcino {\it et al} \cite{Cor1999,Cor2002} and  Gyimesi and Nyul \cite{GyNy}, but we chose to take another route and generalize $S_{(2,0,1)}$ in a way that will fit more naturally in the sense of Coxeter groups, see
Remarks \ref{remark about shifted numbers} and \ref{remark Snkr}
for more elaborate explanations.

Maier \cite{Maier} presented a general notation for triangular numbers:
$$\left| \begin{array}{c}n \\k \end{array}\right| :=\left[\begin{array}{lc|c}
\alpha, & \beta & \gamma \\
\alpha', & \beta' & \gamma'
\end{array}\right]_{n,k},$$
which satisfy the following recursion:
$$\left| \begin{array}{c}n \\k \end{array}\right| =  (\alpha (n-1) + \beta k+\gamma)\left| \begin{array}{c}n-1 \\k \end{array}\right|+(\alpha' (n-1) + \beta' (k-1)+\gamma')\left| \begin{array}{c}n-1 \\k-1 \end{array}\right|$$
This notation contains our Stirling numbers of type $B$ as the special case: $$\alpha=0,\beta=2,\gamma=2r+1,\ \ \alpha'=\beta'=0,\gamma'=1.$$

\medskip
\subsection{Poly-Stirling numbers}
One type of a generalization of the Stirling numbers is obtained by exchanging $k$ in Equation (\ref{kind 2 rec}) or $n-1$ in Equation (\ref{first kind def})  for a value $p(k)$ or $p(n-1)$ (respectively) of a given polynomial $p(x)\in \mathbb{Z}[x]$. In this way, we get the {\it poly-Stirling numbers} (actually one might define the polynomial $p(x)$ also over $\mathbb{C}$ at the expense of losing the nice combinatorial meaning).

We cite here the definitions of poly-Stirling numbers, as they appear in Miceli \cite{miceli11}:
\begin{definition}
Given any nonzero polynomial $p(x)\in \mathbb{Z}[x]$,
define the {\em (unsigned) poly-Stirling numbers of the first kind with respect to $p(x)$}, denoted by $c_{n,k}^{p(x)}$, by the recursion:
$$c_{n,k}^{p(x)}=c_{n-1,k-1}^{p(x)}+p(n-1)c_{n-1,k}^{p(x)},$$
where $c_{0,0}^{p(x)}=1$ and $c_{n,k}^{p(x)}=0$ if either $k>n$ or $k<0$.

In a similar manner, the {\em poly-Stirling numbers of the second kind with respect to $p(x)$}, denoted by $S_{n,k}^{p(x)}$, are defined by the recursion: $$S_{n,k}^{p(x)}=S_{n-1,k-1}^{p(x)}+p(k)S_{n-1,k}^{p(x)},$$
where $S_{0,0}^{p(x)}=1$ and $S_{n,k}^{p(x)}=0$ if either $k>n$ or $k<0$.
\end{definition}

In the current paper, we follow the work of Miceli \cite{miceli11}, who defined two natural types of $q$-analogues of poly-Stirling numbers, namely taking the $q$-analogue of $p(x)$ to be either $[p(x)]_q$ (type II $q$-poly-Stirling) or $p([x]_q)$ (type I $q$-poly-Stirling),
where $[n]_q=1+q+\cdots +q^{n-1}=\frac{q^n-1}{q-1}$ is the $q$-analogue of the number $n$.

We further generalize both Miceli's types $q$-poly-Stirling numbers to $q,r$-poly Stirling numbers, where the additional $r$ stands for the Broder generalization \cite{Broder} mentioned above. Actually, we provide a comprehensive analysis of these numbers by generalizing the theorems of Comtet and Lancaster to the $q,r$-poly-Stirling numbers of both kinds (see Theorem \ref{generalized_comtet for second kind} and Theorem  \ref{generalized_comtet_first_kind}, respectively).

\medskip

Furthermore, we present orthogonality relations between the first and the second kinds  $q,r$-poly-Stirling numbers  of both types I and II.

We also provide some identities pertaining to sum of powers:
$\displaystyle \sum_{j=r}^n \left([p(j)]_q\right)^k$
and
$\displaystyle
\sum_{j=r}^n\bigl(p([j]_q)\bigr)^k$
(see Theorem \ref{th:type2} and Theorem \ref{th:type1} below), which generalize several results regarding
the ordinary sum of powers
$\displaystyle\sum_{j=1}^n j^k$,
see \cite{GZ05,merca18,Schlosser04,warnaar04}.

\subsection{Organization of the paper}
The paper is organized as follows.
Section \ref{type II second kind} deals with the generalization of Comtet's theorem. A comprehensive proof of this theorem is given in the appendix.

Section \ref{sec-combinatorial}
deals with the special case of  Stirling numbers of type $B$ and provides the definition of set partitions of type $B$ and its $r$-version. We then prove some combinatorial identities, including recursions, explicit formulas and the log-concavity property of these numbers.

Section \ref{sec-exponential}
deals with exponential generating functions of these Stirling numbers. These numbers are used as connection constants between two bases of $\mathbb{R}[x]$.

Section \ref{type II first kind} deals with the generalization of Lancaster's theorem for the unsigned $q,r$-poly-Stirling numbers of the first kind.

In Section \ref{orthogonality}, we present the generalization of Lancaster's theorem for the signed $q,r$-poly-Stirling numbers of the first kind, and we prove two applications of this generalization: sum of powers and orthogonality relations.

Section \ref{mixed} deals with some mixed
relations between the two types of $q,r$-poly-Stirling numbers.

\section{Type II  $q,r$-poly-Stirling numbers of the second kind}\label{type II second kind}
In this section, we present a generalization of Comtet's theorem (Theorem \ref{Comtet theorem} above) to the case of type II $q,r$-poly-Stirling numbers of the second kind, mentioned in the introduction (based on Miceli's type II $q$-poly-Stirling number of the second kind \cite{miceli11}).

\begin{definition}
Let $p(x)\in \mathbb{Z}[x]$. The {\em type II $q,r$-poly-Stirling numbers of the second kind} are defined by the recurrence relation:
\begin{equation}
S_{n,k,r}^{p(x)}(q)=S_{n-1,k-1,r}^{p(x)}(q)+[p(k)]_q S_{n-1,k,r}^{p(x)}(q)\quad(r\le k\le n,~n\ge 1)
\label{2st2}
\end{equation}
with $S_{r,r,r}^{p(x)}(q)=1$ and $S_{n,k,r}^{p(x)}(q)=0$ for $k<r$, $k>n$ or $n<r$.
\end{definition}

\medskip

We present here the generalization of Comtet's theorem for type II $q,r$-poly-Stirling numbers of the second kind. For the sake of completeness, the detailed but routine proof can be found in the Appendix.

\begin{theorem}[Generalization of Comtet's theorem for type II]\label{generalized_comtet for second kind}
Let $p(n)$ be a polynomial with non-negative integer coefficients and let $\varphi_k(x), \ k \geq r$, be defined by
$$\varphi_r(x)=1 \mbox{ and } \varphi_k(x)=(x-[p(r)]_q)(x-[p(r+1)]_q) \cdots (x-[p(k-1)]_q) \mbox{ \rm for } k > r.$$
The following are equivalent characterizations for $(S^{p(x)}_{n,k,r}(q))_{n ,k,r\mid  0\leq r \leq k \leq n}$  (where for all other values of the triple $(n,k,r)$, we assume $S^{p(x)}_{n,k,r}(q)=0$):
\begin{enumerate}
\item[(1)] {\bf Defining equation/Change of bases:} For each $n \geq r$: $$x^{n-r}{=} \sum\limits_{k= r}^n S^{p(x)}_{n,k,r}(q)\cdot \varphi_k(x).$$
\item[(2)] {\bf Recursion:} For each $n \geq k > r$:
$$S^{p(x)}_{n,k,r}(q)= S^{p(x)}_{n-1,k-1,r}(q) + [p(k)]_q S^{p(x)}_{n-1,k,r}(q)$$
with the boundary conditions:
$S^{p(x)}_{r,r,r}(q)=1$ and $S_{n,k,r}^{p(x)}(q)=0$ for $k<r$, $k>n$ or $n<r$.
\item[(3)] {\bf Complete recursion:} For $n \geq k > r$:
$$S^{p(x)}_{n,k,r}(q) = \sum\limits_{j=k}^{n} S^{p(x)}_{j-1,k-1,r}(q)\left([p(k)]_q\right)^{n-j},$$
subject to the same boundary conditions as in Condition (2).
\item[(4)] {\bf Ordinary generating function:} For each $k \geq r$:
$$\sum\limits_{n =k}^{\infty} S^{p(x)}_{n,k,r}(q) x^n= \frac{x^k}{(1-[p(r)]_q x) \cdots (1-[p(k)]_q x)}.$$
\item[(5)] {\bf Explicit formula:} For $r\leq k \leq n$:
$$S_{n,k,r}^{p(x)}(q) = \sum\limits_{i_r+i_{r+1}+\cdots +i_k=n-k\atop i_l\geq 0} \left([p(r)]_q\right)^{i_r} \cdots \left([p(k)]_q\right)^{i_k}.$$

\end{enumerate}
\end{theorem}

\bigskip

Condition (5) of Theorem \ref{generalized_comtet for second kind} can be written in an equivalent form, as follows:
\begin{cor}
For $r\leq k\leq n$, we have:
\begin{align*}
S_{n,k,r}^{p(x)}(q)&=\sum_{j_{k-1}=0}^{n-k}\bigl([p(k)]_q\bigr)^{n-k-j_{k-1}}\sum_{j_{k-2}=0}^{j_{k-1}}\bigl([p(k-1)]_q\bigr)^{j_{k-1}-j_{k-2}}\\
&\qquad\qquad\cdots\sum_{j_{r+1}=0}^{j_{r+2}}\bigl([p(r+2)]_q\bigr)^{j_{r+2}-j_{r+1}}\sum_{j_{r}=0}^{j_{r+1}}\bigl([p(r+1)]_q\bigr)^{j_{r+1}-j_{r}}\bigl([p(r)]_q\bigr)^{j_{r}}.
\end{align*}
\label{cor:2st-2-exp}
\end{cor}

\begin{proof}
The right hand side is equivalent to the right hand side of the equation in Condition (5) of Theorem \ref{generalized_comtet for second kind} by the following substitutions: $j_r=i_r$ and\break  $j_m=i_r+i_{r+1}+\cdots +i_m$ for $r<m<k$.
\end{proof}

\medskip

In the next theorem, we present a new way to write $S_{n,k,r}^{p(x)}(q)$, based on its generating function. We start with a lemma:

\begin{lemma}\label{lemma27}
For $n \geq k$ and $\ell \geq 2$, we have:
$$\frac{d^n}{d x^n}\left[\frac{x^k}{(1-a_1 x)\cdots (1-a_\ell x)} \right] = n! \sum_{j=1}^\ell \frac{a_j^{n-k+\ell-1}}{\left[\prod_{i=1,\ i\neq j}^\ell(a_i-a_j)\right](1-a_j x)^{n+1}}.$$
\end{lemma}

\begin{proof} We prove this identity by induction on $n$.
The base case is $n=1$, in which one has to show that:
$$\frac{d}{d x}\left[\frac{x^k}{(1-a_1 x)\cdots (1-a_\ell x)} \right] = \sum_{j=1}^\ell \frac{a_j^{\ell-k}}{\left[\prod_{i=1,\ i\neq j}^\ell(a_i-a_j)\right](1-a_j x)^2}.$$
This equality can be achieved by performing a second induction on $\ell$ and using partial fractions.

The induction step on $n$ follows due to a simple derivation argument.
\end{proof}

\begin{thm}
For $r\le k\le n$, we have
$$
S_{n,k,r}^{p(x)}(q)=\sum_{j=r}^k\frac{\bigl([p(j)]_q\bigr)^{n-r}}{\prod_{i=r,i\ne j}^k\bigl([p(i)]_q-[p(j)]_q\bigr)}\,.
$$
Here, the empty product equals $1$ as usual.
\label{prop:2st2}
\end{thm}

\begin{proof}
The generating function appearing in Condition (4) of Theorem \ref{generalized_comtet for second kind} is:
$$\sum\limits_{n =k}^{\infty} S^{p(x)}_{n,k,r}(q) x^n= \frac{x^k}{(1-[p(r)]_q x) \cdots (1-[p(k)]_q x)}.$$

Then, we get by Lemma \ref{lemma27} for $\ell=k-r+1$ and substituting $[p(i)]_q$ for $a_i$:
\begin{eqnarray*}
S_{n,k,r}^{p(x)}(q)&=&\frac{1}{n!} \cdot \left.\frac{d^n}{d x^n}\left[\frac{x^k}{\prod_{j=r}^k\bigl(1-[p(j)]_q \cdot x\bigr)}\right]\right|_{x=0}=\\
& & \\
&=&\sum_{j=r}^k\frac{\bigl([p(j)]_q\bigr)^{n-r}}{\prod_{i=r,i\ne j}^k\bigl([p(i)]_q-[p(j)]_q\bigr)}.
\end{eqnarray*}
\end{proof}

\section{Combinatorial aspects of Stirling numbers of type $B$}\label{sec-combinatorial}

If we focus on the polynomial $p(k)=2k-1$ and $q=1$ in Equation (\ref{2st2}), we actually deal with the Stirling numbers (of the second kind) of type $B$: $S_{n,k}^B$ and $S_{n,k,r}^B$. In this section, we elaborate on these numbers from the combinatorial point of view.

\subsection{Set partitions of type $B$}\label{second kind}
We now recall the definition of set partitions for type $B$, which are enumerated by these Stirling numbers (see Dolgachev-Lunts \cite[p.~755]{DoLu} and Reiner \cite[Section 2]{R};  implicitly in Dowling \cite{Dow} and Zaslavsky \cite{Za} in the form of signed graphs):

\begin{definition}\label{def of type B par}
A {\em set partition of $[\pm n]=\{\pm 1,\dots, \pm n\}$ of type $B$} is a set partition of  the set $[\pm n]$  such that the following conditions are satisfied:
\begin{itemize}
\item If $B$ appears as a block in the partition, then $-B$ (which is obtained by negating all the elements of $B$) also appears in that partition.
\item There exists at most one block satisfying $-B=B$. This block is called the {\em zero block} (if it exists, it is a subset of $[\pm n]$ of the form $\{\pm i \mid i \in C\}$ for some $C \subseteq [n]$).
\end{itemize}
\end{definition}

For example, the following is a set partition of $[\pm6]$ of type $B$:
$$\{\{1,-1,4,-4\},\{2,3,-5\},\{-2,-3,5\},\{6\},\{-6\}\}.$$

Note that every non-zero block $B$ has a corresponding block $-B$ attached to it. For the sake of convenience, we write for the pair of blocks $B,-B$, only the representative block containing the minimal positive number appearing in $B \cup -B$. For example, the pair of blocks
$B=\{-2,-3,5\},-B=\{2,3,-5\}$ will be represented by the single block $\{2,3,-5\}$.

\begin{definition} \label{definition of Stirling number of type B second kind}
Let $\StirB{n}{k}$ be the number of set partitions of type $B$ having $k$ representative non-zero blocks. This is known as {\em the Stirling number of the second kind of type $B$} (see sequence A085483 in OEIS \cite{OEIS}).
\end{definition}

It is easy to see that $\StirB{n}{n}=\StirB{n}{0}=1$ for each $n\geq 0$.  The following recursion for $\StirB{n}{k}$ is well-known (\cite[Theorem 7; see the Erratum]{Dow}, \cite[Corollary 3]{Beno}, for $m=2$, and \cite[Equation (1)]{Wa}, for $m=2,c=1$):

\begin{prop}\label{recursion theorem}
For each $1 \leq k < n$,
\begin{equation}
\StirB{n}{k}=\StirB{n-1}{k-1}+(2k+1)\StirB{n-1}{k}.
\label{recursion_second_kind_type_b}
\end{equation}
\end{prop}

\medskip

We present here a new generalization of the Stirling number of the second kind of type $B$, based on the work of Broder \cite{Broder} for type $A$. Define the {\it $r$-Stirling number of the second kind of type $B$} as follows:

\begin{definition}\label{special elements}
Let $\StirBr{n}{k}$ be the number of  set partitions of type $B$ of the set $[\pm n]$ into $k$ non-zero blocks such that the numbers $1,\dots,r$ are in distinct non-zero blocks. The numbers $1,\dots,r$ will be called {\em special elements}. Such set partitions will be called {\em $r$-partitions of $[\pm n]$ of type $B$.}
\end{definition}

For example, the following is an element counted by $S_{7,3,2}^B$:
$$\{\{4,-4\},\{\mathbf{1},3,-5\},\{-1,-3,5\},\{\mathbf{2},6\},\{-2,-6\},\{7\},\{-7\}\}.$$

Note that the case $r=0$ brings us back to the definition of the Stirling number of type $B$ given in Definition \ref{definition of Stirling number of type B second kind}.

The recursion for the $r$-Stirling numbers of type $B$ is identical to the one given in Proposition \ref{recursion theorem}, where the only differences are the initial conditions, as the following proposition states:

\begin{prop}\label{recursion r theorem}
If $n<r$, then $\StirBr{n}{k}=0$. If $n=r$, then $\StirBr{n}{k}=\delta_{kr}$. If $n>r$, then: $$\StirBr{n}{k}=\StirBr{n-1}{k-1}+(2k+1)\StirBr{n-1}{k}.$$
\end{prop}

\begin{remark}\label{remark about shifted numbers}
In some places, it is custom to deal with the shifted $r$-Stirling number of type $A$, $S_{n+r,k+r,r}$ (sometimes denoted by $S_{n,k;r}$), which counts the number of set partitions of the set $\{1,\dots,n+r\}$ into $k+r$ blocks such that no two elements of the set $\{1,\dots,r\}$ share a block (see e.g. Mez\H{o}'s book \cite[Section 8.2]{Mezo}).
The defining recursion for these numbers is (see e.g. Bala \cite[Eqn. (14)]{Bala} where $a=1$ and $c=r$, Mansour-Ram\'irez-Shattuck \cite[Eqn. (1.2)]{MRS} where $m=1$, and Maier \cite[Eqn. (1.4)]{Maier} where $a=0$ and $b=1$):
\begin{equation}\label{recursion for r stirling numbers}
S_{n+r,k+r,r}=S_{n+r-1,k+r-1,r}+(k+r)S_{n+r-1,k+r,r}.
\end{equation}

A combinatorial interpretation of  $S_{n,k;r}$ as counting colored set partitions under some restrictions was provided by Gyimesi and Nyul \cite{GyNy}, where the recursion (\ref{recursion for r stirling numbers}) has the form
\begin{equation}\label{recursion for r stirling numbers by nyul}
S_{n+r,k+r,r}=S_{n+r-1,k+r-1,r}+(mk+r)S_{n+r-1,k+r,r},
\end{equation}
where $m$ is the number of colors. This equation previously appeared in Corcino \cite{Cor1999} without combinatorial interpretation, using $\beta$ instead of $m$. Some combinatorial interpretation appeared in Corcino's consequent paper with Aldema \cite{Cor2002}.
These are $r$-Whitney numbers of the lattice of $r$-colored partitions.

Our version for the generalized $r$-Stirling numbers concentrates on the case $m=2$ and corresponds to Whitney numbers of a more natural lattice from the algebraic point of view, that of $r$-signed partitions. The case $r=0$ is isomorphic to the intersection lattice of the roots of the Coxeter group of type $B$ (see also in Section \ref{sect52} below). The corresponding recursion is as follows:
\begin{equation}
S^B_{n+r,k+r,r}=S^B_{n+r-1,k+r-1,r}+(2(k+r)+1)S^B_{n+r-1,k+r,r},
\end{equation}
or in the abbreviated form, using the shifted notation $S^B_{n,k;r}=S^B_{n+r,k+r,r}$:
\begin{equation}
S^B_{n,k;r}=S^B_{n-1,k-1;r}+(2(k+r)+1)S^B_{n-1,k;r}.
\end{equation}
\end{remark}

\subsection{Some combinatorial identities}
In this section, we provide some  combinatorial identities.
We start with the following recursion, connecting Stirling numbers of different values of $r$, which is a generalization for type $B$ of Equation (8.2) in Mez\H{o} \cite{Mezo}:

\begin{prop}
For $n \geq k \geq r$,
\begin{equation}\label {recursion on r}
S^B_{n,k,r}=S^B_{n,k,r-1}-(2r-1)S^B_{n-1,k,r-1}.
\end{equation}
\end{prop}

\begin{proof}
First, we write Equation
(\ref{recursion on r}) in an equivalent way:
$$S^B_{n,k,r-1}-S^B_{n,k,r}=(2r-1)S^B_{n-1,k,r-1}.$$
Now we show this equality in a combinatorial way. The left-hand side is the size of the set containing the set partitions of $[\pm n]$ of type $B$ into $k$ non-zero blocks, such that the first $r-1$ elements are distinguished but the element $r$ is not.

On the right-hand side, we count the same set, by first ignoring the element $r$ and partitioning the $n-1$ remaining elements into $k$ non-zero blocks with $r-1$ distinguished elements and then placing the element $r$ in one of the first $r-1$ blocks in either positive or negative sign or in the zero block, so it will not be considered as a distinguished element.
\end{proof}

The following result connects the generalized Broder's $r$-Stirling numbers of type $B$, $S^B_{n+r,k+r,r}$ (denoted sometimes also as $S^B_{n,k;r}$), with the ordinary Stirling numbers of type $B$, $S^B_{n,k}$. This result can be considered as a type-$B$ generalization of Equation (8.3) in Mez\H{o}'s book \cite{Mezo}:
\begin{prop}
For $n \geq k$ and $r>0$,
\begin{equation}\label{transition from r to ordinary}
S^B_{n+r,k+r,r}=\sum\limits_{j=k}^n
{n \choose j} S^B_{j,k}(2r)^{n-j}.
\end{equation}
\end{prop}

\begin{proof}
Again, we show this equality in a combinatorial way: the left-hand side of Equation (\ref{transition from r to ordinary}) is the size of the set containing the set partitions of $\{1,\dots,n+r\}$ of type $B$ into $k+r$ non-zero blocks, such that the first $r$ elements are distinguished.

The right-hand side describes a different way to count this set: we can assume that the minimal elements of the distinguished blocks are already $1,\dots,r$. Recall that by definition they must be positive. Next, choose the number $j \geq k$  of elements that will not be placed in the distinguished blocks (this choice is expressed by the sum over $j$) and choose these elements out of the remaining $n$ elements in ${n \choose j}$ ways. Next, place these $j$ elements in $k$ non-distinguished blocks in $S^B_{j,k}$ ways (the number $S^B_{j,k}$ includes also the placement of elements in the zero block). Next, arrange the remaining $n-j$ elements, which might be positive or negative, since they are not minimal, in the $r$ distinguished blocks.
\end{proof}

\medskip

Next, we provide an explicit formula for the $r$-Stirling number of type $B$ of the second kind.
We recall the following fact, whose proof is based on the classical inclusion-exclusion argument (see e.g. \cite[Theorem 7.5]{Bona}):
\begin{fact}\label{fact-stirling}
The number of surjective functions from the set $[n]$ to the set $[k]$ is
$$\sum\limits_{\ell=0}^k(-1)^\ell {k \choose \ell}(k-\ell)^n.$$
Therefore, the number of set partitions of the set $[n]$ into $k$ blocks (of type $A$) is:
$$S_{n,k}=\frac{1}{k!}\sum\limits_{\ell=0}^k(-1)^\ell {k \choose \ell}(k-\ell)^n.$$
\end{fact}

In the spirit of the above fact, we present an explicit formula for the Stirling number of type $B$ of the second kind:
\begin{prop}\label{prop inc-exc Snk type B} Let $k \leq n$. Then we have:
\begin{equation}\label{inc-exc Snk type B}
S^B_{n,k}=\frac{1}{2^k k!}\sum\limits_{\ell=0}^k (-1)^{\ell} {k \choose \ell}2^{\ell} \left(\frac{2^{n+1}(k-\ell)^{n+1}-1}{2(k-\ell)-1}\right).
\end{equation}
\end{prop}

\begin{proof}
First, we assume that the zero-block contains exactly $u$ elements, where\break $0 \leq u \leq n$.
Using an argument  similar to the one used to prove Fact \ref{fact-stirling}, we get that the number of set partitions of type $B$ of the set $[\pm n]$ into $k$ non-zero blocks having a zero-block containing exactly $u$ elements, is counted by the following formula:  $$\frac{1}{k!}\sum\limits_{\ell=0}^k(-1)^\ell {k \choose \ell}2^{n-u-(k-\ell)}(k-\ell)^{n-u}.$$
This formula takes into account that all the elements that are not in the zero-block and are not minimal in their blocks can be either positive or negative.

\medskip

Now, allowing $u$ to range from $0$ to $n$, we obtain:
\begin{eqnarray*}
S^B_{n,k} & = & \sum\limits_{u=0}^n \frac{1}{k!}\sum\limits_{\ell=0}^k (-1)^{\ell} {k \choose \ell} 2^{n-u-(k-\ell)} (k-\ell)^{n-u}=\\
&=&\sum\limits_{u=0}^n \frac{1}{k!}\sum\limits_{\ell=0}^k (-1)^{\ell} {k \choose \ell} \frac{2^n}{2^k \cdot 2^u}2^{\ell} \frac{(k-\ell)^n}{(k-\ell)^u}=\\
&=&\frac{1}{ 2^k k!}\sum_{\ell=0}^k (-1)^{\ell} {k \choose \ell} 2^n \cdot 2^{\ell} (k-\ell)^n \sum_{u=0}^n \frac{1}{(2(k-\ell))^u}=\\
&=&\frac{1}{2^k k!}\sum\limits_{\ell=0}^k (-1)^{\ell} {k \choose \ell}2^n \cdot 2^{\ell} (k-\ell)^n\left(\frac{2^{n+1}(k-\ell)^{n+1}-1}{2^n (k-\ell)^n(2(k-\ell)-1)}\right)=\\
&=&\frac{1}{2^k k!}\sum\limits_{\ell=0}^k (-1)^{\ell} {k \choose \ell} 2^{\ell} \left(\frac{2^{n+1}(k-\ell)^{n+1}-1}{ 2(k-\ell)-1}\right).
\end{eqnarray*}
\end{proof}

We now prove an $r$-version for this result. We start by a formulation for type $A$, which already appeared in Corcino \cite{Cor1999}, with a combinatorial proof based on the inclusion-exclusion principle (similar proof can be found in Corcino and Aldema \cite{Cor2002}):

\begin{prop}\label{prop inc-exc Snkr type A}
 Let $n,r \in \mathbb{N}$ and let $k\leq n$. Then:
\begin{equation}\label{eqn Snkr type A}
S_{n+r,k+r,r}=\frac{1}{k!}\sum\limits_{\ell=0}^k{(-1)^{k-\ell}{k \choose \ell}(\ell+r)^n}.
\end{equation}

\end{prop}

\begin{proof}
Assuming that for each $1\leq i \leq r$, the $i$-th block already contains the element $i$, we are left with the problem of counting the number of functions from the set $[n]$ to the set $\{1,\dots,r+k\}$ satisfying that the set $\{r+1, \dots, r+k\}$ is contained in their image (similar to the functions which we dealt with in Fact \ref{fact-stirling}).
\end{proof}

Now we move to the $r$-version for type $B$:

\begin{prop}\label{exclusion-inclusion with r}  Let $k \leq n$. Then we have:
\begin{equation}\label{eqn Snkr type B}
S^B_{n+r,k+r,r}=\frac{1}{2^k k!}\sum\limits_{\ell=0}^k (-1)^{\ell} {k \choose \ell}2^{\ell} \left(\frac{2^{n+1}(k+r-\ell)^{n+1}-1}{2(k+r-\ell)-1}\right).
\end{equation}

\end{prop}

\begin{proof}
Apply the arguments of Propositions \ref{prop inc-exc Snk type B} and \ref{prop inc-exc Snkr type A} to get the following expression
\begin{equation}\label{exclusion-inclusion expression for B and r}
\frac{1}{k!}\sum\limits_{\ell=0}^k(-1)^\ell {k \choose \ell}2^{n-u-(k-\ell)}(k+r-\ell)^{n-u}
\end{equation}
instead of Equation (\ref{inc-exc Snk type B}).

This formula is based on the inclusion-exclusion principle, where we have to insert the $n$ elements: $r+1,\dots,r+n$ into $k+r$ parts. Moreover,
it takes into account the fact that the elements $1,\dots, r$ are already positive and minimal in their respective blocks, and hence only $k$ more elements should be signed positively.
\end{proof}

\begin{remark}\label{remark Snkr}
A formula similar to the one appearing in Equations (\ref{eqn Snkr type A}) and (\ref{eqn Snkr type B}) can be found in the work of Corcino et al \cite{Cor1999,Cor2002} and Gyimesi and Nyul \cite{GyNy}. Note however that they deal with a different definition of set partitions of type $B$: they do not have a zero-block and their $r$ first blocks (which contain the elements $1,\dots,r$ respectively) have no negative
elements, see \cite[beginning of Section 3]{Cor2002} and \cite[Definition 2.2]{GyNy}.   We have elaborated on the significance of our choice of set partitions of type $B$ in Section \ref{section intro Broder}.
Note that it seems that there is a missing condition in the combinatorial interpretation of $(r,\beta)$-Stirling number \cite[beginning of Section 3]{Cor2002} (this missing condition does appear in \cite[Definition 2.2]{GyNy}).
\end{remark}

\subsection{Log-concavity of $S_{n,k,r}^B$}
We recall that a sequence of real numbers $(a_i)_{i\in \mathbb{N}}$ is called {\it log-concave} if the following condition holds for each $i \geq 2$:
$$a_i^2 \geq a_{i-1}a_{i+1}.$$ It is well known that the sequence of Stirling numbers of the second kind (as well as of the first kind) is log-concave. Different proofs were given to these facts (see for example \cite{Harper,JG}).  It seems that the simplest proof for the log-concavity of ordinary Stirling numbers of the second kind was provided by Sagan \cite[Theorem 1]{Sagan1988}. We present Sagan's statement here and we use it to prove the log-concavity for type $B$:
\begin{thm}[Sagan \cite{Sagan1988}]\label{sagan criterion}
Let $t_{n,k}$ be an extended triangular array, satisfying $t_{n,k} = c_{n,k}t_{n-1,k-1}+d_{n,k}t_{n-1,k}$ for all $n \geq 1$,
where $t_{n,k}, c_{n,k}$ and $d_{n,k}$ are all non-negative integers. Suppose that:
\begin{itemize}
    \item[(i)] $c_{n,k}$ and $d_{n,k}$ are log-concave in $k$,
\item[(ii)] $c_{n,k-1} d_{n,k+1}+ c_{n,k+1}d_{n,k-1}\leq 2c_{n,k}d_{n,k}$ \ for all $n \geq 1$.
\end{itemize}
Then $t_{n,k}$ is log-concave in $k$.
\end{thm}

We have the following result:
\begin{prop}
Let $n,r$ be such that $r \leq n$. The sequence of $r$-Stirling numbers of type $B$ of the second kind, $S^B_{n,k,r}$, is log-concave.
\end{prop}

\begin{proof}
The recursion for the $r$-Stirling numbers of type $B$ of the second kind, $S^B_{n,k,r}$, is:
$$\StirBr{n}{k}=\StirBr{n-1}{k-1}+(2k+1)\StirBr{n-1}{k}.$$
Hence, we have: $c_{n,k}=1$ and $d_{n,k} =2k+1$, which satisfy the conditions of Theorem \ref{sagan criterion}.
\end{proof}

\section{Exponential generating functions for Stirling numbers of type $B$}\label{sec-exponential}

In this section, we deal with the exponential generating functions of $\StirB{n}{k}$ and $\StirBr{n}{k}$ (Sections \ref{sect 51} and \ref{sect52} respectively). In Section \ref{sect connection constants}, we introduce an application of the exponential generating function of $\StirBr{n}{k}$ as connection constants between bases of $\mathbb{R}[x]$.

\subsection{An exponential generating function of $\StirB{n}{k}$}\label{sect 51}
We have the following result concerning an exponential generating function of $\StirB{n}{k}$:

\begin{prop}\label{prop_expo_gene_function}
For $k \geq0$,
\begin{equation}
\sum\limits_{n=k}^{\infty}\StirB{n}{k}\frac{z^n}
 {n!}=\frac{1}{k!2^k}e^z(e^{2z}-1)^k.
\label{eq_exponential}
\end{equation}
\end{prop}

The proof we provide here uses the symbolic approach to compute exponential generating functions for labeled structures due to Flajolet and Sedgewick (see \cite[Section II.1]{Flaj}).
For the sake of self-containment, we list the following labeled structures, most of which can be found in \cite{Flaj} there:

\begin{itemize}
\item Let $\mathcal{Z}=\{\textcircled{\tiny{1}}\}$ be the atomic labeled class having only one element. The corresponding exponential generating function (EGF) is $E(z)=z$.

\item We denote by $2\Z$ the direct sum $\Z \oplus \Z$, with corresponding EGF: $E(z)=2z$.

\item For a labeled class $\mathcal{B}$, let {\sc Seq}$_k(\B)$ be the class containing all the sequences of elements from the class $\B$ of length $k$. Provided that the EGF of $\B$ is denoted by $B(z)$, the EGF of {\sc Seq}$_k(\B)$  is $E(z)=B(z)^k$.

\item For a labeled class $\B$, let \set$(\B)$ be the class of sets of elements taken from $\B$ and let \setk$(\B)$ be the class of $k$-sets of elements taken from $\B$. If the EGF of $\B$ is denoted by $B(z)$, then the EGF of \set$(\B)$ is $e^{B(z)}$, while the EGF of \set$_k(\B)$ is $\frac{1}{k!}B(z)^k$.

\item In particular, \set$(\Z)$ is the class of urns (or totally disconnected graphs), with corresponding EGF given by $E(z)=e^z$.

\item Define $\mbox{\sset}(\Z)=\mbox{\set}(2\Z)$ and interpret it as the class of signed sets of integers, where by a {\it signed set} we mean a set in which every element is signed by either $+$ or $-$. It is easy to see that the EGF of \sset$(\Z)$ is $e^{2z}$.
As usual, the class \sset$_{\geq 1}(\Z)$ denotes the set of the non-empty signed sets of \sset \ and its EGF is thus $E(z)=e^{2z}-1$.

\item Define an equivalence relation $R$ on \sset$_{\geq 1}(\Z)$ that identifies two signed sets $A$ and $B$ if $A=-B$ (the negation is taken elementwise).
For example, the signed sets $\{2,-3,5,-7\}$ and $\{-2,3,-5,7\}$ are equivalent. The first set is used as a representative, since its minimal element is positive.

The quotient class will be denoted by $\sfrac{\mbox{\sset}_{\geq 1}(\Z)}{R}$ and by its definition, it can be easily observed that its corresponding EGF is $\frac{e^{2z}-1}{2}$.
\end{itemize}

Now we are ready to supply a symbolic proof for Proposition \ref{prop_expo_gene_function}:
\begin{proof}[Proof of  Proposition \ref{prop_expo_gene_function}]
Every set partition of $[\pm n]$ of type $B$ counted by $S^B_{n,k}$ is composed of a zero block (which might be empty) followed by $k$ signed non-empty blocks, every minimal element of which is positive. Therefore, the combinatorial class we are counting can be written as
$$\mbox{\set}(\Z) \cdot \mbox{\set}_k (\sfrac{\mbox{\sset}_{\geq 1}(\Z)}{R}),$$
and hence its EGF is $e^z\cdot \frac{1}{k!}\left(\frac{e^{2z}-1}{2}\right)^k$ as claimed.
\end{proof}

Sagan and Swanson \cite[Theorem 4.3]{SaSw} provided a $q$-analogue of Proposition \ref{prop_expo_gene_function} using $q$-exponentials.

\medskip

Summing over $k$, we immediately get the following corollary, which is a special case of Bala \cite[Eqn. (4)]{Bala} for $(a,b,c)=(2,0,1)$:
\begin{cor}\label{coro-for-bala}
$$\sum\limits_{k,n}S^B(n,k)\frac{z^n}{n!}t^k=e^{z}e^{(e^{2z}-1)\frac{t}{2}}.$$
\end{cor}

\subsection{An exponential generating function for $\StirBr{n+r}{k+r}$}\label{sect52}
We have the following result concerning an exponential generating function for the shifted $r$-Stirling numbers of type $B$ $\StirBr{n+r}{k+r}$:

\begin{theorem}\label{thm_exponential}
Let $k,r\in \mathbb{N}$. Then:
$$\displaystyle\sum\limits_{n=k}^{\infty}\StirBr{n+r}{k+r}\frac{z^n}{n!}=\frac{1}{k!2^k}e^{(2r+1)z}(e^{2z}-1)^k.$$
\end{theorem}

We prove it in two different ways: the first proof is analytical in nature, but is based on a combinatorial lemma, that is itself a generalization of a result appearing in Mez\H{o}'s book \cite[Section 8.2]{Mezo}. The second proof is similar to the one given for Proposition \ref{prop_expo_gene_function} and uses the symbolic approach for exponential generating functions of Flajolet and Sedgewick \cite[Section II.1]{Flaj}.

\medskip

The analytical proof is based on the following combinatorial identity:
\begin{lemma}\label{lemma_exponential}
Let $n,k,r \in \mathbb{N}$. Then
\begin{equation}\label{equ lemma exponential}
\StirBr{n}{k}=\sum\limits_{m=0}^{n-r}{{n-r} \choose m} \cdot S^B_{n-r-m, k-r} \cdot (2r)^m.
\end{equation}
\end{lemma}

\begin{proof}
Recall that the set partitions of type $B$ counted by $\StirBr{n}{k}$, have $r$ distinguished blocks, each of them contains exactly one of the distinguished elements $1,\dots,r$, but might contains additional elements from $r+1,\dots,n$.

Let $m$ be the number of additional elements that will occupy the $r$ distinguished blocks (note that the absolute value of each of these $m$ elements is greater than $r$).
The number ${n-r \choose m}$ counts the ways to choose $m$ such elements.

Then, we insert the remaining $n-r-m$ elements in $k-r$ additional blocks in  $S^B_{n-r-m, k-r}$ ways. Next, we fill the $r$ distinguished blocks with $r+m$ elements in $S^B_{r+m ,r}$ ways.
Note, however, that $S^B_{r+m,r}=(2r)^m$, since each one of the additional $m$ elements can be located in one of the $r$ distinguished blocks, and can be either positive or negative.
This gives us the required equality.
\end{proof}

Based on this lemma, one can easily deduce an expression for the exponential generating function of $\StirBr{n+r}{k+r}$:

\begin{proof}[Analytical proof of Theorem \ref{thm_exponential}]
By substituting in Equation (\ref{equ lemma exponential}) $n+r$ and $k+r$  for $n$ and $k$ respectively, we get:
$$\StirBr{n+r}{k+r}=\sum\limits_{m=0}^n{n \choose m}S^B_{ n-m,  k} \cdot(2r)^m\stackrel{m\leftarrow n-m}{=}\sum\limits_{m=0}^n{n \choose m}\cdot (2r)^{n-m} \cdot \StirB{m}{k}.$$
Since the exponential generating function of the (ordinary) Stirling numbers for type $B$ is: $\displaystyle\sum\limits_{n=k}^{\infty}\StirB{n}{k}\frac{z^n}{n!}=\frac{1}{k!2^k}e^z(e^{2z}-1)^k$ by Proposition \ref{prop_expo_gene_function} above, we conclude from the rule of product of exponential generating functions (see e.g. \cite[Prop. 5.1.1]{EC2}) that
\begin{eqnarray*}
\sum\limits_{n=k}^{\infty}\StirBr{n+r}{k+r}\frac{z^n}{n!}&=&\sum\limits_{n=k}^{\infty}\sum\limits_{m=0}^n{n \choose m}\cdot (2r)^{n-m} \cdot \StirB{m}{k}\frac{z^n}{n!}=\\
& = & e^{2rz} \cdot \frac{1}{k!2^k}e^z(e^{2z}-1)^k=\frac{1}{k!2^k}e^{(2r+1)z}(e^{2z}-1)^k,
\end{eqnarray*}
as required.
\end{proof}

Next, we present the second proof of Theorem \ref{thm_exponential} based on the symbolic approach:

\begin{proof}[Symbolic proof of Theorem \ref{thm_exponential}]
Every $r$-partition of $\{1,\dots,n+r\}$ of type $B$ (see Definition \ref{special elements}) counted by $\StirBr{n+r}{k+r}$ is composed of a zero block which is taken from the set of numbers $\{r+1,\dots,n+r\}$, followed by a sequence of $r$ signed sets (each one is preoccupied with a minimal element from the set $\{1,\dots,r\}$) and finally a set of $k$ signed sets such that the minimal element of each set is positive.

In symbols, this description amounts to:
$$\mbox{\set}(\Z)
\cdot \mbox{\sc Seq}_r(\mbox{\sset}(\Z))
\cdot  \mbox{\set}_k(\sfrac{\mbox{\sset}_{\geq 1}(\Z)}{R}),$$ which contributes
$$e^z \cdot e^{2rz} \cdot \frac{1}{k!2^k}(e^{2z}-1)^k=\frac{1}{k!2^k}e^{(2r+1)z}(e^{2z}-1)^k,$$ as claimed.
\end{proof}

Note that by substituting $r=0$ in Theorem \ref{thm_exponential}, we get again Proposition \ref{prop_expo_gene_function}.

\medskip

The following corollary is a different presentation to a special case of Bala \cite[Eqn. (4)]{Bala} for $(a,b,c)=(2,0,r)$:

\begin{cor}\label{coro-for-bala-r}
$$\sum\limits_{k,n \geq 0}^{\infty}\StirBr{n+r}{k+r}\frac{z^n}{n!}t^k=e^{(2r+1)z}e^{t \frac{e^{2z}-1}{2}}.$$
\end{cor}

\subsection{An application to connection constants}\label{sect connection constants}

Given $n,r \in \mathbb{N}$, let $\{(x+2r)^n\}_{n\in \mathbb{N}} $ and $\{(x-1)(x-3)\cdots (x-2n+1)\}_{n\in \mathbb{N}}$
be two sequences of polynomials, forming two different bases of $\mathbb{R}[x]$. The following result presents the numbers $\StirBr{n+r}{k+r}$ as connection constants between these two bases of $\mathbb{R}[x]$:
\begin{thm}\label{thm4.3}
Let $n,r \in \mathbb{N}$. Then we have: $$(x+2r)^n=\sum\limits_{k=0}^n \StirBr{n+r}{k+r}\cdot (x-1)(x-3)\cdots (x-2k+1).$$
\end{thm}

\begin{proof}
We calculate:
\begin{eqnarray*}
\hspace{-10pt}\sum\limits_{n=0}^{\infty}(2x+2r)^n\frac{z^n}{n!}& = & e^{2(x+r)z}=e^{(2r+1)z}e^{2\left(x-\frac{1}{2}\right)z}= e^{(2r+1)z}(1+(e^{2z}-1))^{\left(x-\frac{1}{2}\right)}=\\
&=& \sum\limits_{k=0}^{\infty}e^{(2r+1)z}(e^{2z}-1)^k \cdot \frac{2^k\left(x-\frac{1}{2}\right)\left(x-\frac{3}{2}\right)\cdots \left(x-k+\frac{1}{2}\right)}{2^k k!}=\\ &\stackrel{{\rm Thm.}\ \ref{thm_exponential}}{=}& \sum\limits_{k=0}^{\infty} \left( \sum\limits_{n=0}^{\infty}\StirBr{n+r}{k+r}\frac{z^n}{n!} \right) (2x-1)(2x-3)\cdots (2x-2k+1) =\\
&=& \sum\limits_{n=0}^{\infty}  \left(  \sum\limits_{k=0}^n \StirBr{n+r}{k+r} \cdot (2x-1)(2x-3)\cdots (2x-2k+1)\right) \frac{z^n}{n!}
\end{eqnarray*}
By comparing the coefficients of $z^n$, we have: $$(2x+2r)^n =\sum\limits_{k=0}^n \StirBr{n+r}{k+r} \cdot (2x-1)(2x-3)\cdots (2x-2k+1).$$
By substituting $x$ for $2x$, the result follows.
\end{proof}

Note that Corollary 2.4 of Sagan and Swanson \cite{SaSw} is a $q$-analogue for the $r=0$ case.

\section{The type II $q,r$-poly-Stirling numbers of the first kind}\label{type II first kind}

In this section, we move to the Stirling numbers of the first kind, and present a generalization of Lancaster's theorem (Theorem \ref{Lancaster} above) to the case of unsigned type II $q,r$-poly-Stirling numbers of the first kind, mentioned in the introduction (based on Miceli's type II $q$-poly-Stirling number of the first kind \cite{miceli11}).

\begin{definition}
Let $p(x)\in \mathbb{Z}[x]$. The {\em (unsigned) type II $q,r$-poly-Stirling numbers of the first kind}, denoted $c_{n,k,r}^{p(x)}(q)$, are defined by the recurrence relation:
\begin{equation}
c_{n,k,r}^{p(x)}(q)=c_{n-1,k-1,r}^{p(x)}(q)+[p(n-1)]_q c_{n-1,k,r}^{p(x)}(q)\quad(r\le k\le n,~n\ge 1)
\label{2st1}
\end{equation}
with $c_{r,r,r}^{p(x)}(q)=1$ and $c_{n,k,r}^{p(x)}(q)=0$ (for $k<r$, $k>n$ or $n<r$).
\end{definition}

When $r=0$, the recurrence (\ref{2st1}) is reduced to that in Miceli \cite[Eqn. (39)]{miceli11}.
Moreover, the boundary conditions for $r=0$: $c_{0,0,0}^{p(x)}(q)=1$ and $c_{n,k,0}^{p(x)}(q)=0$ (for $k<0$, $k>n$ or $n<0$) are compatible with those of Miceli.

\medskip

Next, we extend and generalize Lancaster's result to the case of type II unsigned $q,r$-poly-Stirling numbers of the first kind:

\begin{theorem}[Generalized Lancaster's theorem for $q,r$-poly-Stirling numbers] \label{generalized_comtet_first_kind}
Let $p(n)$ be a polynomial with non-negative integer coefficients.
The following are equivalent characterizations for $(c^{p(x)}_{n,k,r}(q))_{n ,k,r\ \mid\  0\leq r \leq k \leq n}$  (where for all other values of the triple $(n,k,r)$, we assume $c^{p(x)}_{n,k,r}(q)=0$):
\begin{enumerate}
\item[(1)] {\bf Defining equation/generating function:} $$(x+[p(r)]_q)(x+[p(r+1)]_q) \cdots (x+[p(n-1)]_q) = \sum\limits_{k=r}^n c_{n,k,r}^{p(x)}(q)\cdot x^{k-r}.$$
\item[(2)] {\bf Recursion:} For each $n\geq k>r$:
$$c^{p(x)}_{n,k,r}= c_{n-1,k-1,r}^{p(x)}(q) + [p(n-1)]_q\  c_{n-1,k,r}^{p(x)}(q)$$
with the boundary conditions:
$c_{n,r,r}^{p(x)}(q)=[p(r)]_q [p(r+1)]_q \cdots [p(n-1)]_q$ and $c_{r,k,r}^{p(x)}(q)=\delta_{kr}$.
\item[(3)] {\bf Complete Recursion:} For $n \geq k >r$:
$$c_{n,k,r}^{p(x)}(q) = \sum\limits_{j=k}^{n} c_{j-1,k-1,r}^{p(x)}(q)\prod\limits_{i=j}^{n-1}[p(i)]_q,$$
subject to the same boundary conditions as in Condition (2).

\item[(4)] {\bf Explicit formula:} For $n \geq k \geq r$:
$$c_{n,k,r}^{p(x)}(q)=\prod_{j=r}^{n-1}[p(j)]_q\sum_{r\le i_{r+1}<\dots<i_{k}\le n-1}\frac{1}{[p(i_{r+1})]_q\cdots [p(i_{k})]_q}.$$
\end{enumerate}
\end{theorem}

\begin{proof} We prove only the implications {\bf [(2)$\Longrightarrow$(4)]} and {\bf [(4)$\Longrightarrow$(2)]}, since the other implications can be easily deduced from the proof of the original theorem of  Lancaster \cite{Lancaster}.

\medskip

\noindent
{\bf [(2)$\Longrightarrow$(4)]:}
We prove the identity by induction on $n$.
The base case is: $n=k=r$. Indeed,
 $$c_{r,r,r}^{p(x)}(q)  =\prod_{j=r}^{r-1}[p(j)]_q =1$$
as required.

\medskip

Assume correctness for $n-1$, and we prove the correctness for $n$: by the recursion (\ref{2st1}), we have:
\begin{eqnarray*}
c_{n,k,r}^{p(x)}(q)
&\stackrel{(\ref{2st1})}{=}&c_{n-1,k-1,r}^{p(x)}(q)+[p(n-1)]_q\ c_{n-1,k,r}^{p(x)}(q)=\\
& \stackrel{\rm Assumption}{=} &\left[\prod_{j=r}^{n-2}[p(j)]_q\right]\sum_{r\le i_{r+1}<\dots<i_{k-1}\le n-2}\frac{1}{[p(i_{r+1})]_q\cdots [p(i_{k-1})]_q}+\\
& & \qquad +[p(n-1)]_q\left[\prod_{j=r}^{n-2}[p(j)]_q\right]\sum_{r\le i_{r+1}<\dots<i_{k}\le n-2}\frac{1}{[p(i_{r+1})]_q\cdots [p(i_{k})]_q}=\\
&=&\left[\prod_{j=r}^{n-1}[p(j)]_q\right]\sum_{r\le i_{r+1}<\dots<i_{k-1}\le n-2}\frac{1}{[p(i_{r+1})]_q\cdots [p(i_{k-1})]_q[p(n-1)]_q}+\\
& & \qquad+\left[\prod_{j=r}^{n-1}[p(j)]_q\right]\sum_{r\le i_{r+1}<\dots<i_{k}\le n-2}\frac{1}{[p(i_{r+1})]_q\cdots [p(i_{k})]_q}=\\
&=&\left[\prod_{j=r}^{n-1}[p(j)]_q\right]\sum_{r\le i_{r+1}<\dots<i_{k}\le n-1}\frac{1}{[p(i_{r+1})]_q\cdots [p(i_{k})]_q}\,.
\end{eqnarray*}

\medskip

\noindent
{\bf [(4)$\Longrightarrow$(2)]:}
Assume that we have Condition (4):
$$c_{n,k,r}^{p(x)}(q)=\prod_{j=r}^{n-1}[p(j)]_q\sum_{r\le i_{r+1}<\dots<i_{k}\le n-1}\frac{1}{[p(i_{r+1})]_q\cdots [p(i_{k})]_q},$$
then it is immediate that
$c_{n,r,r}^{p(x)}(q)=[p(r)]_q [p(r+1)]_q \cdots [p(n-1)]_q$ and $c_{r,k,r}^{p(x)}(q)=\delta_{kr}$, and therefore the boundary conditions are satisfied.

Now we show the recurrence relation in Condition (2):
\begin{eqnarray*}
c_{n,k,r}^{p(x)}(q) &\stackrel{{\rm Cond. }\ (4)}{=}& \left[\prod_{j=r}^{n-1}[p(j)]_q\right]\sum_{r\le i_{r+1}<\dots<i_{k}\le n-1}\frac{1}{[p(i_{r+1})]_q\cdots [p(i_{k})]_q}=\\
& = & \left[\prod_{j=r}^{n-2}[p(j)]_q\right]\sum_{r\le i_{r+1}<\dots<i_{k-1}\le n-2}\frac{1}{[p(i_{r+1})]_q\cdots [p(i_{k-1})]_q} +\\ & & \qquad+[p(n-1)]_q \left(\left[\prod_{j=r}^{n-2}[p(j)]_q\right]\sum_{r\le i_{r+1}<\dots<i_{k}\le n-2}\frac{1}{[p(i_{r+1})]_q\cdots [p(i_{k})]_q}\right)=\\
& \stackrel{{\rm Cond.}\ (4)}{=} & c_{n-1,k-1,r}^{p(x)}(q) + [p(n-1)]_q\ c_{n-1,k,r}^{p(x)}(q),
\end{eqnarray*}
as required.
\end{proof}

\section{Orthogonality relations between poly-Stirling numbers of both kinds}\label{orthogonality}

In the previous section, we focused on the unsigned version of the $q,r$-poly-Stirling number of the first kind $c_{n,k,r}^{p(x)}$. In the current section, we discuss its signed version, in order to present two of its applications: an identity related to a sum of powers (Section \ref{sect 71}) and the orthogonality relations between the Stirling numbers of first and second kinds (Section \ref{sect 72}).

\begin{definition}
Let $p(x) \in \mathbb{Z}[x]$. The {\em type II (signed) $q,r$-poly-Stirling numbers of the first kind}, denoted $s_{n,k,r}^{p(x)}(q)$, are defined by the recurrence relation:
\begin{equation}
s_{n,k,r}^{p(x)}(q)=s_{n-1,k-1,r}^{p(x)}(q)-[p(n-1)]_q\  s_{n-1,k,r}^{p(x)}(q)\quad(r\le k\le n,~n\ge 1)
\label{2st1sign}
\end{equation}
with $s_{r,r,r}^{p(x)}(q)=1$ and $s_{n,k,r}^{p(x)}(q)=0$ (for $k<r$, $k>n$ or $n<r$).
\end{definition}

We present here the generalization of Lancaster's theorem for the signed case, the correctness of which is evident by substituting $x \to -x$ in the formulation of the
unsigned case (Theorem \ref{generalized_comtet_first_kind}):

\begin{theorem}[Generalized Lancaster's theorem for signed type II $q,r$-poly-Stirling numbers of the first kind] \label{generalized_comtet_first_kind_signed}
Let $p(x)\in \mathbb{Z}[x]$. The following are equivalent characterizations for $(s^{p(x)}_{n,k,r}(q))_{n ,k,r\ \mid \ 0\leq r \leq k \leq n}$  (where for all other values of the triple $(n,k,r)$, we assume $s^{p(x)}_{n,k,r}(q)=0$):
\begin{enumerate}
\item[(1)] {\bf Defining equation/generating function:} $$(x-[p(r)]_q)(x-[p(r+1)]_q) \cdots (x-[p(n-1)]_q){=} \sum\limits_{k=r}^n s_{n,k,r}^{p(x)}(q)\cdot x^{k-r}.$$
\item[(2)] {\bf Recursion:} For each $n\geq k>r$:
$$s^{p(x)}_{n,k,r}= s_{n-1,k-1,r}^{p(x)}(q) - [p(n-1)]_q\  s_{n-1,k,r}^{p(x)}(q),$$
with the boundary conditions:
$$s_{n,r,r}^{p(x)}(q)=(-1)^{n-r}[p(r)]_q [p(r+1)]_q \cdots [p(n-1)]_q$$ and $s_{r,k,r}^{p(x)}(q)=\delta_{kr}$.
\item[(3)] {\bf Complete recursion:} For $n \geq k >r$:
$$s_{n,k,r}^{p(x)}(q) = \sum\limits_{j=k}^{n} (-1)^{n-j} s_{j-1,k-1,r}^{p(x)}(q)\prod\limits_{i=j}^{n-1}[p(i)]_q,$$
subject to the same boundary conditions as in Condition (2).

\item[(4)] {\bf Explicit formula:}
For $n\geq k \geq r$:
$$s_{n,k,r}^{p(x)}(q)=(-1)^{n-k}\left[\prod_{j=r}^{n-1}[p(j)]_q\right]\sum_{r\le i_{r+1}<\dots<i_{k}\le n-1}\frac{1}{[p(i_{r+1})]_q\cdots [p(i_{k})]_q}.$$
\end{enumerate}
\end{theorem}

\subsection{An application to sum of powers}\label{sect 71}

In this subsection, we provide an application of Theorem \ref{generalized_comtet_first_kind_signed} to the sum of powers of the expressions $[p(j)]_q$: 
\begin{theorem}
For integers $n,k$ and $r$ with $n\ge r$ and $k\ge 1$, we have:
$$
\sum_{j=r}^n\bigl([p(j)]_q\bigr)^k=-\sum_{\ell=1}^{n-r+1}\ell s_{n+1,n+1-\ell,r}^{p(x)}(q) S_{n+k-\ell,n,r}^{p(x)}(q)\,.
$$
\label{th:type2}
\end{theorem}

\begin{proof}
By Condition (1) of Theorem \ref{generalized_comtet_first_kind_signed} for $n+1$ instead of $n$ we have:
\begin{equation}
x^{r}\prod_{j=r}^{n}\bigl(x-[p(j)]_q\bigr)=\sum_{k=r}^{n+1} s_{n+1,k,r}^{p(x)}(q) x^k.
\label{eq:2st1}
\end{equation}
Next, we substitute $\frac{1}{t}$ for $x$:
$$\left(\frac{1}{t}\right)^{r}\prod_{j=r}^{n}\left(\frac{1}{t}-[p(j)]_q\right)=\sum_{k=r}^{n+1} s_{n+1,k,r}^{p(x)}(q) \left(\frac{1}{t}\right)^k.$$
Multiplying this by $t^{n+1}$, yields:
\begin{equation}
F_1(t):=\prod_{j=r}^{n}\bigl(1-t [p(j)]_q \bigr)=\sum_{k=r}^{n+1} s_{n+1,k,r}^{p(x)}(q)\ t^{n-k+1}\stackrel{n-k+1 \rightarrow \ell}{=}\sum_{\ell=0}^{n-r+1}s_{n+1,n+1-\ell,r}^{p(x)}\ t^\ell\,. \label{F1(t)}
\end{equation}

Recall from Theorem \ref{generalized_comtet for second kind}(4) that
$$\sum\limits_{n =k}^{\infty} S^{p(x)}_{n,k,r}(q) x^n= \frac{x^k}{(1-x[p(r)]_q ) \cdots (1-x [p(k)]_q )}.$$
By the substitutions (in this order): $n \rightarrow n+\nu$, $k \rightarrow n$ and $x \rightarrow t$, and dividing both sides by $t^n$, we get:
$$\sum\limits_{\nu =0}^{\infty} S^{p(x)}_{n+\nu,n,r}(q)\ t^{\nu}= \frac{1}{(1-t [p(r)]_q ) \cdots (1-t [p(n)]_q )}=\prod_{j=r}^n\bigl(1-t [p(j)]_q \bigr)^{-1}=:F_2(t).$$

Note that $F_1(t)F_2(t)=1$. Now, we have:
\begin{equation}
\frac{d}{d t}\log F_2(t)=\frac{d}{d t}\left[\sum_{j=r}^n -\log\left(1-t[p(j)]_q\right)\right]=\sum_{j=r}^n\frac{[p(j)]_q}{1-t[p(j)]_q}=\sum_{j=r}^n\sum_{k=1}^\infty\bigl([p(j)]_q\bigr)^k t^{k-1}.
\label{eq:221}
\end{equation}
Note that:
$$\frac{d}{dt}F_1(t)=\frac{d}{dt}\left[\prod_{j=r}^{n}\bigl(1-t[p(j)]_q \bigr)\right]=\sum_{m=r}^n -[p(m)]_q \prod_{j=r, j\neq m}^{n}\bigl(1-t[p(j)]_q \bigr).$$
Therefore, we have
\begin{equation}
-\left(\frac{d}{dt}F_1(t)\right)F_2(t)=\sum_{m=r}^n\frac{[p(m)]_q}{1-t[p(m)]_q}\stackrel{(\ref{eq:221})}{=}\frac{d}{d t}\log F_2(t).
\label{eq:222}
\end{equation}

On the other hand, by Equation (\ref{F1(t)}), we also have:
$$
\frac{d}{dt} F_1(t)=\sum_{\ell=1}^{n-r+1}\ell s_{n+1,n+1-\ell,r}^{p(x)}(q)t^{\ell-1}\,.
$$

Combining Equations (\ref{eq:221}) and (\ref{eq:222}), we have:
$$
-\left(\sum_{\ell=1}^{n+1}\ell s_{n+1,n+1-\ell,r}^{p(x)}(q)t^{\ell-1}\right)\left(\sum_{\nu=0}^\infty S_{n+\nu,n,r}^{p(x)}(q)t^\nu\right)=\sum_{k=1}^\infty\sum_{j=r}^n\bigl([p(j)]_q\bigr)^k t^{k-1},$$
which can be also written as:
$$
-\sum_{\ell=1}^{n+1}\sum_{\nu=0}^\infty \ell s_{n+1,n+1-\ell,r}^{p(x)}(q) S_{n+\nu,n,r}^{p(x)}(q)t^{\ell -1+\nu}=\sum_{k=1}^\infty\sum_{j=r}^n\bigl([p(j)]_q\bigr)^k t^{k-1}.$$

Comparing the coefficients of $t^{k-1}$ in both sides (imposing that $\nu=k-\ell$), yields the requested formula.
\end{proof}

Note that Theorem 1.1 of Merca \cite{merca18} is a special case of Theorem \ref{th:type2} where $r=0$ and $p(x)=x$.

\subsection{An application to orthogonality relations}\label{sect 72}

As a consequence of Theorems \ref{generalized_comtet for second kind} and \ref{generalized_comtet_first_kind_signed}, we get the orthogonality relations between the type II $q,r$-poly-Stirling numbers of the first and second kinds:

\begin{theorem}\label{orthogonality result}
For $r\le\ell\le n$ and $r<n$, we have:
$$\sum_{k=\ell}^n S_{n,k,r}^{p(x)}(q) s_{k,\ell,r}^{p(x)}(q)=\delta_{n\ell}$$
$$\sum_{k=\ell}^n s_{n,k,r}^{p(x)}(q) S_{k,\ell,r}^{p(x)}(q)=\delta_{n\ell}.$$
\end{theorem}

\begin{proof}
By Theorem \ref{generalized_comtet for second kind}(1), we have:
$$x^{n-r}=\sum_{k=r}^n S_{n,k,r}^{p(x)}(q)\prod_{j=r+1}^k\bigl(x-[p(j-1)]_q\bigr).$$
On the other hand, by Condition (1) of Theorem \ref{generalized_comtet_first_kind_signed} we have:
$$\prod_{j=r}^{n-1}\left(x-[p(j)]_q\right) = \sum\limits_{k=r}^n s_{n,k,r}^{p(x)}(q)\cdot x^{k-r}.$$
By substituting either one of these equations into the other and changing the order of summations, we readily obtain the orthogonality relations.
\end{proof}

\section{Type I $q,r$-poly-Stirling numbers of both kinds and mixed relations between both types}\label{mixed}

In this section, we present another version of $q,r$-poly-Stirling numbers of the first and second kinds. Based on this version, we present some mixed relations between both types of these numbers.

\subsection{Type I $q,r$-poly-Stirling numbers of both kinds}

Here we present the generalization of Miceli's  Type I $q$-poly-Stirling numbers of the first and second kind (see \cite{miceli11}):
\begin{definition}
The {\em type I $q,r$-poly-Stirling numbers of the second kind} are defined by the recurrence:
\begin{equation}
\overline{S}_{n,k,r}^{p(x)}(q)=\overline{S}_{n-1,k-1,r}^{p(x)}(q)+p([k]_q)\overline{S}_{n-1,k,r}^{p(x)}(q)\quad(r\le k\le n,~n\ge 1),
\label{1st2}
\end{equation}
with $\overline{S}_{r,r,r}^{p(x)}(q)=1$ and $\overline{S}_{n,k,r}^{p(x)}(q)=0$ for $k<r$, $k>n$ or $n<r$.
\end{definition}

\begin{definition}
The {\em type I signed $q,r$-poly-Stirling numbers of the first kind} are defined by the recurrence:
\begin{equation}
\overline{s}_{n,k,r}^{p(x)}(q)=\overline{s}_{n-1,k-1,r}^{p(x)}(q)-p([n-1]_q)\overline{s}_{n-1,k,r}^{p(x)}(q)\quad(r\le k\le n,~n\ge 1)
\label{1st1}
\end{equation}
with $\overline{s}_{r,r,r}^{p(x)}(q)=1$ and $\overline{s}_{n,k,r}^{p(x)}(q)=0$ for $k<0$, $k>n$ or $n<r$.
\end{definition}

\begin{definition}
The {\em type I unsigned $q,r$-poly-Stirling numbers of the first kind} are defined by the recurrence:
\begin{equation}
\overline{c}_{n,k,r}^{p(x)}(q)=\overline{c}_{n-1,k-1,r}^{p(x)}(q)+p([n-1]_q)\overline{c}_{n-1,k,r}^{p(x)}(q)\quad(r\le k\le n,~n\ge 1)
\label{1st1-1-unsigned}
\end{equation}
with $\overline{c}_{r,r,r}^{p(x)}(q)=1$ and $\overline{c}_{n,k,r}^{p(x)}(q)=0$ for $k<0$, $k>n$ or $n<r$.
\end{definition}

All the results of Sections \ref{type II second kind}, \ref{type II first kind} and \ref{orthogonality} and their proofs can be transferred verbatim to type I as well  (the only change is the location of the squared brackets);
nevertheless, we provide here the formulation of the result parallel to Theorem \ref{th:type2}, 
for our use in the next subsection:
\begin{theorem}
For integers $n$ and $k$ with $n\ge r$ and $k\ge 1$, we have
$$
\sum_{j=r}^n\bigl(p([j]_q)\bigr)^k=-\sum_{\ell=1}^{n-r+1}\ell \overline{s}_{n+1,n+1-\ell,r}^{p(x)}(q) \overline{S}_{n+k-\ell,n,r}^{p(x)}(q)\,.
$$
\label{th:type1}
\end{theorem}

\subsection{Mixed relations}
We can obtain a mixed relation of both types of $q,r$-poly-Stirling numbers:
\begin{theorem}
Let $u \geq 1$ be a positive integer. For $n\ge r$, we have
$$
\sum_{\ell=1}^{n-r+1}\ell s_{n+1,n+1-\ell,r}^{x}(q) S_{n+ku-\ell,n,r}^{x}(q)
=\sum_{\ell=1}^{n-r+1}\ell \overline{s}_{n+1,n+1-\ell,r}^{x^u}(q)\overline{S}_{n+k-\ell,n,r}^{x^u}(q)\,.
$$
\label{th:mix}
\end{theorem}

\begin{proof}
If we put $p(x)=x$ in Theorem \ref{th:type2}, we get:
$$\sum_{j=r}^n\bigl([j]_q\bigr)^k=-\sum_{\ell=1}^{n-r+1}\ell s_{n+1,n+1-\ell,r}^{x}(q) S_{n+k-\ell,n,r}^{x}(q).$$
Substituting $ku$ for $k$, we get:
\begin{equation}
\sum_{j=r}^n\bigl([j]_q\bigr)^{ku}=-\sum_{\ell=1}^{n-r+1}\ell s_{n+1,n+1-\ell,r}^{x}(q) S_{n+ku-\ell,n,r}^{x}(q).
\label{7.3modified}\end{equation}

On the other hand, if we put  $p(x)=x^u$ in Theorem \ref{th:type1}, we get:
\begin{equation}
\sum_{j=r}^n\bigl(([j]_q)^u\bigr)^k=-\sum_{\ell=1}^{n-r+1}\ell \overline{s}_{n+1,n+1-\ell,r}^{x^u}(q) \overline{S}_{n+k-\ell,n,r}^{x^u}(q).
\label{8.4modified}
\end{equation}

Comparing Equations (\ref{7.3modified}) and (\ref{8.4modified}) yields the requested equality.
\end{proof}

\section*{Acknowledgements}
We thank Bruce Sagan and Josh Swanson for many fruitful discussions, including correcting our formula in Theorem \ref{thm_exponential}.

\section{Appendix - Proof of Theorem \ref{generalized_comtet for second kind}}

\begin{lemma}\label{second kind special cases}
\begin{enumerate}
\item[(1)] $S_{n,n,r}^{p(x)}(q)=1 $
\item[(2)] $S_{n,n-1,r}^{p(x)}(q) = \sum\limits_{i=r}^{n-1} [p(i)]_q$
\item[(3)] $S_{n,r,r}^{p(x)}(q)=\bigl([p(r)]_q\bigr)^{n-r}$
\end{enumerate}
Hence, in general, if $p(r) \neq 0$, then $S_{n,r,r}^{p(x)}(q)\ne 0$ even if $r=0$.
\end{lemma}

\begin{proof}
By the recurrence relation (\ref{2st2}), we get:
\begin{enumerate}
\item[(1)] $$S_{n,n,r}^{p(x)}(q)=S_{n-1,n-1,r}^{p(x)}(q)+[p(n)]_q \underbrace{S_{n-1,n,r}^{p(x)}(q)}_{=0}=S_{n-1,n-1,r}^{p(x)}(q)=\cdots=S_{r,r,r}^{p(x)}(q)=1$$
\item[(2)]
\begin{eqnarray*}
S_{n,n-1,r}^{p(x)}(q)&=&S_{n-1,n-2,r}^{p(x)}(q)+[p(n-1)]_q \cdot \underbrace{S_{n-1,n-1,r}^{p(x)}(q)}_{=1}=\cdots=\\
&=&\underbrace{S_{r,r-1,r}^{p(x)}(q)}_{=0}+[p(r)]_q+\cdots+[p(n-2)]_q+[p(n-1)]_q=\sum\limits_{i=r}^{n-1} [p(i)]_q
\end{eqnarray*}

\item[(3)]
$$S_{n,r,r}^{p(x)}(q)=\underbrace{S_{n-1,r-1,r}^{p(x)}(q)}_{=0}+[p(r)]_q S_{n-1,r,r}^{p(x)}(q)=\cdots=\bigl([p(r)]_q\bigr)^{n-r}\underbrace{S_{r,r,r}^{p(x)}(q)}_{=1}=\bigl([p(r)]_q\bigr)^{n-r}$$
\end{enumerate}

\end{proof}

\begin{lemma}\label{second kind basis k=r+1}
For $n \geq r+2$:
$$S_{n,r+1,r}^{p(x)}(q)=\sum_{i_r+i_{r+1}=n-r-1\atop i_r,i_{r+1}\ge 0}\bigl([p(r)]_q\bigr)^{i_r}\bigl([p(r+1)]_q\bigr)^{i_{r+1}}$$
\end{lemma}

\begin{proof}
We prove this lemma by induction on $n$.
For the base case, substitute $n=r+2$ in Lemma \ref{second kind special cases}(2):
$$S_{r+2,r+1,r}^{p(x)}(q)=[p(r)]_q+[p(r+1)]_q.$$

Assume correctness for $n-1$ and we prove for $n>r+2$ using the recurrence relation in Equation (\ref{2st2}):
\begin{eqnarray*}
S_{n,r+1,r}^{p(x)}(q) & \stackrel{(\ref{2st2})}{=} & S_{n-1,r,r}^{p(x)}(q)+[p(r+1)]_q S_{n-1,r+1,r}^{p(x)}(q)=\\
& \stackrel{\ref{second kind special cases}(3) + {\rm Assumption}}{=} & \bigl([p(r)]_q\bigr)^{n-r-1}+[p(r+1)]_q\sum_{i_r+i_{r+1}=n-r-2\atop i_r,i_{r+1}\ge 0}\bigl([p(r)]_q\bigr)^{i_r}\bigl([p(r+1)]_q\bigr)^{i_{r+1}}=\\
& = & \sum_{i_r+i_{r+1}=n-r-1\atop i_r,i_{r+1}\ge 0}\bigl([p(r)]_q\bigr)^{i_r}\bigl([p(r+1)]_q\bigr)^{i_{r+1}}
\end{eqnarray*}
\end{proof}

When $r=0$, the recurrence relation (\ref{2st2}) is reduced to the one appearing in Miceli \cite[Equation (37)]{miceli11}.
Moreover, the boundary conditions for $r=0$,
namely $S_{0,0,0}^{p(x)}(q)=1$ and $S_{n,k,0}^{p(x)}(q)=0$ (for $k<0$, $k>n$ or $n<0$), are compatible with those of Miceli.

\medskip

Before proving the theorem, we need one more lemma:

\begin{lemma}\label{long calc}
Using the notations of Theorem \ref{generalized_comtet for second kind}, we have for $r \leq n$:
$$x\sum_{k=r}^{n-1}S_{n-1,k,r}^{p(x)}(q)\cdot\varphi_k(x)=\sum_{k=r}^n\left(S_{n-1,k-1,r}^{p(x)}(q)+[p(k)]_q S_{n-1,k,r}^{p(x)}(q) \right)\varphi_k(x)$$
\end{lemma}

\begin{proof}
It is easy to check that for $n=r$ both sides are $0$. We compute for $n>r$:
{\small \begin{eqnarray*}
x\sum_{k=r}^{n-1}S_{n-1,k,r}^{p(x)}(q)\cdot\varphi_k(x)\notag
&=&\sum_{k=r}^{n-1}\left(x-[p(k)]_q+[p(k)]_q\right) S_{n-1,k,r}^{p(x)}(q)\cdot\varphi_k(x)=\\
&=&\sum_{k=r}^{n-1}S_{n-1,k,r}^{p(x)}(q)\cdot\left(x-[p(k)]_q\right)\varphi_k(x)+\sum_{k=r}^{n-1}[p(k)]_q S_{n-1,k,r}^{p(x)}(q)\cdot\varphi_k(x)=\\
&=&\sum_{k=r}^{n-1}S_{n-1,k,r}^{p(x)}(q)\cdot\varphi_{k+1}(x)+\sum_{k=r}^{n-1}[p(k)]_q S_{n-1,k,r}^{p(x)}(q)\cdot\varphi_k(x)=\\
&\stackrel{k+1 \to k\ {\rm in\ left\ sum}}{=}&\sum_{k=r+1}^{n}S_{n-1,k-1,r}^{p(x)}(q)\cdot\varphi_{k}(x)+\sum_{k=r}^{n-1}[p(k)]_q S_{n-1,k,r}^{p(x)}(q)\cdot\varphi_k(x)=\\
&=&\sum_{k=r}^n \left(S_{n-1,k-1,r}^{p(x)}(q)+[p(k)]_q S_{n-1,k,r}^{p(x)}(q)\right)\varphi_k(x)\notag\\
& & -\underbrace{S_{n-1,r-1,r}^{p(x)}(q)}_{=0}\varphi_{r}(x)-\underbrace{S_{n-1,n,r}^{p(x)}(q)}_{=0}[p(n)]_q \varphi_n(x)= \notag\\
&=&
\sum_{k=r}^n\left(S_{n-1,k-1,r}^{p(x)}(q)+[p(k)]_q S_{n-1,k,r}^{p(x)}(q) \right)\varphi_k(x)
\end{eqnarray*}}
as requested.
\end{proof}

\begin{proof}[Proof of Theorem \ref{generalized_comtet for second kind}]
The proof consists of the following eight parts.

\medskip

\noindent
{\bf [(1)$\Longrightarrow$(2)]:}
We first prove the boundary condition:
$$1=x^{r-r}\stackrel{\rm Cond.\ (1)}{=}\sum\limits_{k=r}^{r}{S^{p(x)}_{r,k,r}(q)\varphi_k(x)}=S^{p(x)}_{r,r,r}(q)\underbrace{\varphi_r(x)}_{=1}=S^{p(x)}_{r,r,r}(q),$$
so we get: $S^{p(x)}_{r,r,r}(q)=1$ as needed.

Now, we prove the recurrence relation. We have:
$$
\sum_{k=r}^n S_{n,k,r}^{p(x)}(q)\cdot\varphi_k(x)\stackrel{\rm Cond.\ (1)}{=}x^{n-r}=x\cdot x^{n-1-r}\stackrel{\rm Cond.\ (1)}{=}x\sum_{k=r}^{n-1}
S_{n-1,k,r}^{p(x)}(q)\cdot\varphi_k(x) .
$$
By Lemma \ref{long calc}, we can replace the right-hand-side of the above equation by the right-hand-side of the equation in the lemma, in order to get:
$$
\sum_{k=r}^n S_{n,k,r}^{p(x)}(q)\cdot\varphi_k(x)=\sum_{k=r}^n\left(S_{n-1,k-1,r}^{p(x)}(q)+[p(k)]_q S_{n-1,k,r}^{p(x)}(q)\right)\varphi_k(x) .
$$
Comparing the coefficients of the basis elements $\{\varphi_k(x)\}$ on both sides, we obtain:
\begin{equation*}
S_{n,k,r}^{p(x)}(q)=S_{n-1,k-1,r}^{p(x)}(q)+[p(k)]_q S_{n-1,k,r}^{p(x)}(q).
\end{equation*}

\medskip

\noindent
{\bf [(2)$\Longrightarrow$(1)]:}
Expand the recurrence relation appearing in Condition (2) and use Lemma \ref{long calc} $n-r$ times. Explicitly:
\begin{eqnarray*}
\sum\limits_{k=r}^n S_{n,k,r}^{p(x)}(q)\cdot\varphi_k(x)&  \stackrel{\rm Cond.\ (2)}= &  \sum\limits_{k=r}^n \left( S_{n-1,k-1,r}^{p(x)}(q)+[p(k)]_q S_{n-1,k,r}^{p(x)}(q) \right)\cdot\varphi_k(x) =\\
& \stackrel{\rm Lemma\ \ref{long calc}}= & x\sum\limits_{k=r}^{n-1} S_{n-1,k,r}^{p(x)}(q)\cdot\varphi_k(x) =\\
&\stackrel{{\rm Cond.}\ (2)+ \ref{long calc}}{=}&x^2\sum\limits_{k=r}^{n-2}S_{n-2,k,r}^{p(x)}(q)\cdot\varphi_k(x) =\cdots \stackrel{{\rm Cond.}\ (2)+ \ref{long calc}}{=}\\
&=&x^{n-r}\sum\limits_{k=r}^{r}S_{r,k,r}^{p(x)}(q)\cdot\varphi_k(x)=\\
&=&x^{n-r}\underbrace{S_{r,r,r}^{p(x)}(q)\cdot\varphi_r(x)}_{=1}\stackrel{\rm Bound.\ Cond.\ (2)}=x^{n-r}.
\end{eqnarray*}

\medskip

\noindent
{\bf [(2)$\Longrightarrow$(3)]:}
Expand the recurrence relation appearing in Condition (2)  $n-k$ times. Explicitly:
\begin{eqnarray*}
S_{n,k,r}^{p(x)}(q)&=&S_{n-1,k-1,r}^{p(x)}(q)+[p(k)]_q S_{n-1,k,r}^{p(x)}(q)=\\
&=&S_{n-1,k-1,r}^{p(x)}(q)+[p(k)]_q\bigl(S_{n-2,k-1,r}^{p(x)}(q)+[p(k)]_q S_{n-2,k,r}^{p(x)}(q)\bigr)=\\
&=&S_{n-1,k-1,r}^{p(x)}(q)+[p(k)]_q S_{n-2,k-1,r}^{p(x)}(q)+\\
& & \ +\bigl([p(k)]_q\bigr)^2(S_{n-3,k-1,r}^{p(x)}(q)+[p(k)]_q S_{n-3,k,r}^{p(x)}(q)\bigr)  = \cdots = \\
&=&S_{n-1,k-1,r}^{p(x)}(q)+[p(k)]_q S_{n-2,k-1,r}^{p(x)}(q)+\bigl([p(k)]_q\bigr)^2 S_{n-3,k-1,r}^{p(x)}(q)+\\
& & \ +\cdots+\bigl([p(k)]_q\bigr)^{n-k}S_{k-1,k-1,r}^{p(x)}(q)=\sum\limits_{j=k}^{n} S^{p(x)}_{j-1,k-1,r}(q)\left([p(k)]_q\right)^{n-j}.\\
\end{eqnarray*}

\medskip

\noindent
{\bf [(3)$\Longrightarrow$(2)]:}
We check for $n=r+1$ separately: in this case, the only possible value of $k$ is $r+1$. By Condition (3), we have:
\begin{eqnarray*}S_{r+1,r+1,r}^{p(x)}(q)&=&\sum_{j=r+1}^{r+1} S_{j-1,r,r}^{p(x)}(q)\bigl([p(k)]_q\bigr)^{r+1-j}= S_{r,r,r}^{p(x)}(q)=1,
\end{eqnarray*}
while the recursion in Condition (2), achieves the same value:
\begin{eqnarray*}S_{r+1,r+1,r}^{p(x)}(q)&=& \underbrace{S_{r,r,r}^{p(x)}(q)}_{=1}+ [p(r)]_q \underbrace{S_{r,r+1,r}^{p(x)}(q)}_{=0}=1.
\end{eqnarray*}

For $n>r+1$, apply Condition (3) for both $n,n-1>r$:
\begin{equation}
S_{n,k,r}^{p(x)}(q)=\sum_{j=k}^n S_{j-1,k-1,r}^{p(x)}(q)\bigl([p(k)]_q\bigr)^{n-j}\,,\label{cond3-n}
\end{equation}
\begin{equation}S_{n-1,k,r}^{p(x)}(q)=\sum_{j=k}^{n-1}S_{j-1,k-1,r}^{p(x)}(q)\bigl([p(k)]_q\bigr)^{n-j-1}\,.\label{cond3n-1}
\end{equation}
Hence,
\begin{eqnarray*}
S_{n,k,r}^{p(x)}(q)&\stackrel{ (\ref{cond3-n})}{=}&\sum_{j=k}^n S_{j-1,k-1,r}^{p(x)}(q)\bigl([p(k)]_q\bigr)^{n-j}=\\
& = & S_{n-1,k-1,r}^{p(x)}(q)+[p(k)]_q\sum_{j=k}^{n-1}S_{j-1,k-1,r}^{p(x)}(q)\bigl([p(k)]_q\bigr)^{n-j-1}=\\
&\stackrel{(\ref{cond3n-1})}{=}&S_{n-1,k-1,r}^{p(x)}(q)+[p(k)]_qS_{n-1,k,r}^{p(x)}(q),
\end{eqnarray*}
which is the recurrence relation in Condition (2).

\medskip

\noindent
{\bf [(2)$\Longrightarrow$(4)]:}
We assume that the recurrence relation in Condition (2) holds. Denote $\displaystyle f_k:=\sum\limits_{n=k}^\infty S_{n,k,r}^{p(x)}(q)x^n$ for each $k \geq r$.
Then, for $k>r$, we have:
\begin{eqnarray*}
f_k&=&\sum_{n=k}^\infty S_{n,k,r}^{p(x)}(q)x^n=\\
&\stackrel{{\rm Cond.}\ (2)}{=}&\sum_{n=k}^\infty\left(S_{n-1,k-1,r}^{p(x)}(q)+[p(k)]_q S_{n-1,k,r}^{p(x)}(q)\right)x^n=\\
&=&\sum_{n=k}^\infty S_{n-1,k-1,r}^{p(x)}(q)x^n+\sum_{n=k}^\infty[p(k)]_q S_{n-1,k,r}^{p(x)}(q)x^n=\\
&\stackrel{n\leftarrow n-1}{=}&x\sum_{n=k-1}^\infty S_{n,k-1,r}^{p(x)}(q)x^n+[p(k)]_q \cdot x \sum_{n=k-1}^\infty S_{n,k,r}^{p(x)}(q)x^n=\\
&\stackrel{\left(S_{k-1,k,r}^{p(x)}(q)=0\right)}{=}&x\sum_{n=k-1}^\infty S_{n,k-1,r}^{p(x)}(q)x^n+[p(k)]_q\cdot x \sum_{n=k}^\infty S_{n,k,r}^{p(x)}(q)x^n=\\
&=&x f_{k-1}+[p(k)]_q \cdot x f_k.
\end{eqnarray*}
This implies:
\begin{equation}f_k= \frac{x}{1-[p(k)]_q \cdot x} \cdot f_{k-1}\label{fk}.\end{equation}
For $k=r$, we have:
$$ f_r=\sum_{n=r}^\infty S_{n,r,r}^{p(x)}(q)x^n\stackrel{{\rm Lemma}\ \ref{second kind special cases}(3)}{=}\sum_{n=r}^\infty\bigl([p(r)]_q\bigr)^{n-r}x^n=x^r\sum_{n=r}^\infty\bigl([p(r)]_q \cdot x\bigr)^{n-r}=\frac{x^r}{1-[p(r)]_q \cdot x}.$$
In summary, using Equation (\ref{fk}), we get for $k \geq r$:
$$\sum_{n=k}^\infty S_{n,k,r}^{p(x)}(q)x^n=f_k=\frac{x^k}{\bigl(1-[p(r)]_q \cdot x\bigr)\bigl(1-[p(r+1)]_q \cdot x\bigr)\cdots\bigl(1-[p(k)]_q \cdot x\bigr)}\,. $$

\medskip

\noindent
{\bf [(4)$\Longrightarrow$(2)]:} Denote for $k \geq r$:
$$g_k:=\frac{x^k}{\bigl(1-[p(r)]_q \cdot x\bigr)\bigl(1-[p(r+1)]_q \cdot x\bigr)\cdots\bigl(1-[p(k)]_q \cdot x\bigr)}=\sum_{n=k}^\infty S_{n,k,r}^{p(x)}(q)x^n\,.$$
As \
$ g_r=\frac{x^r}{1-[p(r)]_q \cdot x}$, we get:
$g_k=\frac{x}{1-[p(k)]_q \cdot x} \cdot g_{k-1}$.
Thus for $k \geq r$:
\begin{eqnarray*}
\sum_{n=k}^\infty S_{n,k,r}^{p(x)}(q)x^n & = & g_k = x g_{k-1}+[p(k)]_q x g_k= \\
&=&x\sum_{n=k-1}^\infty S_{n,k-1,r}^{p(x)}(q)x^n+[p(k)]_q x\sum_{n=k}^\infty S_{n,k,r}^{p(x)}(q)x^n=\\
&\stackrel{S_{k-1,k,r}^{p(x)}(q)=0}{=}&\sum_{n=k-1}^\infty S_{n,k-1,r}^{p(x)}(q)x^{n+1}+[p(k)]_q \sum_{n=k-1}^\infty S_{n,k,r}^{p(x)}(q)x^{n+1}=\\
&\stackrel{n+1 \rightarrow n}{=}&\sum_{n=k}^\infty\left(S_{n-1,k-1,r}^{p(x)}(q)+[p(k)]_q S_{n-1,k,r}^{p(x)}(q)\right)x^n.
\end{eqnarray*}
Comparing the coefficients on both sides, we obtain:
$$
S_{n,k,r}^{p(x)}(q)=S_{n-1,k-1,r}^{p(x)}(q)+[p(k)]_q S_{n-1,k,r}^{p(x)}(q)\,.
$$

\medskip

\noindent
{\bf [(2)$\Longrightarrow$(5)]:}
Before proving the general case, we have to verify some specific cases.
For $r \leq k=n$, we have by the boundary condition in Condition (2) and Lemma \ref{second kind special cases}(1):
$$S_{n,n,r}^{p(x)}(q) =1= \sum\limits_{i_r+i_{r+1}+\cdots +i_k=0\atop i_l\geq 0} \left([p(r)]_q\right)^{i_r} \cdots \left([p(k)]_q\right)^{i_k}$$
as required.

For $r=k<n$, we have by Lemma \ref{second kind special cases}(3):
$$S_{n,r,r}^{p(x)}(q) =  \left([p(r)]_q\right)^{n-r}$$
again as required.

\medskip

Now, we pass to the general case, where $r<k<n$. We start by induction on $n$. The base case is $n=r+2$. Since $r<k<n$, it means that $k=r+1$, and this is the case $n=r+2$ of Lemma \ref{second kind basis k=r+1} .

Now assume its correctness for $n-1\geq r+2$ and we prove it for $n$.
This will be done by a proof for all values of $k$ in the range $\{r+1,\dots, n-1\}$. The first case is $k=r+1$, which was proved in Lemma \ref{second kind basis k=r+1}. For $k>r+1$, we  prove it using the recurrence in Condition (2):
\begin{eqnarray*}
S_{n,k,r}^{p(x)}(q)&\stackrel{{\rm Cond.} \ (2)}{=}&S_{n-1,k-1,r}^{p(x)}(q)+[p(k)]_q S_{n-1,k,r}^{p(x)}(q)=\\
&\stackrel{\rm Assumption}{=} &\sum_{i_{r}+i_{r+1}+\cdots+i_{k-1}=n-k\atop i_{r},i_{r+1},\dots,i_{k-1}\ge 0}\bigl([p(r)]_q\bigr)^{i_{r}}\bigl([p(r+1)]_q\bigr)^{i_{r+1}}\cdots\bigl([p(k-1)]_q\bigr)^{i_{k-1}}+\\
& & +[p(k)]_q\sum_{i_{r}+i_{r+1}+\cdots+i_k=n-1-k\atop i_{r},i_{r+1},\dots,i_k\ge 0}\bigl([p(r)]_q\bigr)^{i_{r}}\bigl([p(r+1)]_q\bigr)^{i_{r+1}}\cdots\bigl([p(k)]_q\bigr)^{i_k}= \\
&=&\sum_{i_{r}+i_{r+1}+\cdots+i_k=n-k\atop i_{r},i_{r+1},\dots,i_k\ge 0}\bigl([p(r)]_q\bigr)^{i_{r}}\bigl([p(r+1)]_q\bigr)^{i_{r+1}}\cdots\bigl([p(k)]_q\bigr)^{i_k}.
\end{eqnarray*}

\medskip

\noindent
{\bf [(5)$\Longrightarrow$(2)]:}
Assume that Condition (5) is satisfied:
$$S_{n,k,r}^{p(x)}(q) = \sum\limits_{i_r+i_{r+1}+\cdots +i_k=n-k\atop i_l\geq 0} \left([p(r)]_q\right)^{i_r} \cdots \left([p(k)]_q\right)^{i_k}.$$
Then it is immediate that
$S^{p(x)}_{r,r,r}(q)=1$ and $S_{n,k,r}^{p(x)}(q)=0$ for $k<r$, $k>n$ or $n<r$, and therefore the boundary conditions are satisfied.

Now we have to show the recurrence relation in Condition (2):
\begin{eqnarray*}
S_{n,k,r}^{p(x)}(q) &\stackrel{{\rm Cond.} \ (5)}{=}& \sum\limits_{i_r+i_{r+1}+\cdots +i_k=n-k\atop i_l\geq 0} \left([p(r)]_q\right)^{i_r} \cdots \left([p(k)]_q\right)^{i_k}=\\
& = & \sum\limits_{i_r+i_{r+1}+\cdots +i_{k-1}=n-k\atop i_l\geq 0} \left([p(r)]_q\right)^{i_r} \cdots \left([p(k-1)]_q\right)^{i_{k-1}} +\\ & & \qquad+[p(k)]_q \sum\limits_{i_r+i_{r+1}+\cdots +i_k=n-k-1\atop i_l\geq 0} \left([p(r)]_q\right)^{i_r} \cdots \left([p(k)]_q\right)^{i_k}=\\
& \stackrel{{\rm Cond.} \ (5)}{=} & S_{n-1,k-1,r}^{p(x)}(q) + [p(k)]_q S_{n-1,k,r}^{p(x)}(q),
\end{eqnarray*}
which is the requested recurrence relation.
\end{proof}

\end{document}